\documentclass[11pt]{article}
\usepackage[cm]{fullpage}
\usepackage[small]{titlesec}
\usepackage[cm]{fullpage}

\usepackage{color}

\usepackage{amsmath,amscd, float,rotating}
\usepackage{pb-diagram}

\usepackage{latexsym}
\usepackage{amsfonts}
\usepackage{amsmath}
\usepackage{amssymb}

\numberwithin{equation}{section}
\newcommand{\qed}{\hfill \ensuremath{\Box}}

\def\XXint#1#2#3{{\setbox0=\hbox{$#1{#2#3}{\int}$}
\vcenter{\hbox{$#2#3$}}\kern-.5\wd0}}

\newcommand{\tr}[2]{\textrm{tr}_{#1} \, {#2}}
\newcommand{\ve}{\varepsilon}

\newcommand{\Oorb}{\Omega_{\textrm{orb}}}
\newcommand{\oorb}{\omega_{\textrm{orb}}}

\newcommand{\ofs}{\omega_{\textrm{FS}}}
\newcommand{\osm}{\omega_{\infty}}

\newcommand{\dbar}{\overline{\partial}}

\newcommand{\ddt}[1]{\frac{\partial #1}{\partial t}}

\newcommand{\ov}[1]{\overline{#1}}
\newcommand{\un}[1]{\underline{#1}}
\newcommand{\Xc}{X_{\textrm{can}}}

\newcommand{\geu}{g_{\textrm{Eucl}}}
\newcommand{\Cone}{\mathcal{C}}
\newcommand{\oX}{\omega_X}

\newcommand{\dds}{\frac{d}{ds}}
\newcommand{\emax}{\eta_{\textrm{max}}}
\newcommand{\diam}{\mathrm{diam}}

\newcommand{\ddbar}{\frac{\sqrt{-1}}{2\pi} \partial\dbar}

\newcommand{\oeuc}{\omega_{\textrm{Eucl}}}

\begin{document}
\newcounter{remark}
\newcounter{theor}
\setcounter{remark}{0} \setcounter{theor}{1}
\newtheorem{claim}{Claim}
\newtheorem{theorem}{Theorem}[section]
\newtheorem{proposition}{Proposition}[section]
\newtheorem{lemma}{Lemma}[section]
\newtheorem{definition}{Definition}[section]
\newtheorem{conjecture}{Conjecture}[section]
\newtheorem{corollary}{Corollary}[section]
\newenvironment{proof}[1][Proof]{\begin{trivlist}
\item[\hskip \labelsep {\bfseries #1}]}{\end{trivlist}}
\newenvironment{remark}[1][Remark]{\addtocounter{remark}{1} \begin{trivlist}
\item[\hskip \labelsep {\bfseries #1
\thesection.\theremark}]}{\end{trivlist}}
\newenvironment{example}[1][Example]{\addtocounter{remark}{1} \begin{trivlist}
\item[\hskip \labelsep {\bfseries #1
\thesection.\theremark}]}{\end{trivlist}}
~

\centerline{\bf \Large Contracting exceptional divisors by the K\"ahler-Ricci flow II\footnote{The first-named author is supported in part by an NSF CAREER grant 
  DMS-08-47524  and the second-named author by the grant DMS-08-48193.  Both authors are also supported in part by Sloan Research Fellowships. }}

\bigskip
\bigskip

\centerline{\large \bf Jian Song$^{*}$ and
Ben Weinkove$^\dagger$}


\bigskip
\bigskip
\noindent
{\bf Abstract} \  
We investigate the case of the K\"ahler-Ricci flow blowing down disjoint
 exceptional divisors with normal bundle $\mathcal{O}(-k)$ to orbifold points.  We prove smooth convergence outside the exceptional divisors and global Gromov-Hausdorff convergence.   In addition, we establish the result that the Gromov-Hausdorff limit coincides with the metric completion of the limiting metric under the flow.  This improves and extends the previous work of the authors.   We apply this to $\mathbb{P}^1$-bundles which are higher-dimensional analogues of the Hirzebruch surfaces.  We also consider the case of a minimal surface of general type with only distinct irreducible $(-2)$-curves and show that solutions to  the normalized K\"ahler-Ricci flow converge in the  Gromov-Hausdorff sense to a K\"ahler-Einstein orbifold.



\bigskip
\bigskip


\section{Introduction}

The K\"ahler-Ricci flow, suitably normalized, converges to a K\"ahler-Einstein metric on a compact manifold $X$ with negative or zero first Chern class \cite{Y1, A, Cao}.  It was shown in \cite{P2, TZhu} that on a manifold with positive first Chern class the flow converges to a K\"ahler-Einstein metric when one exists.  Convergence is also known under certain weaker assumptions (see \cite{CW, MS, PS, PSSW, Sz, To, Zhu} and the discussions therein).  

There has been considerable interest in investigating the behavior of the K\"ahler-Ricci flow on more general algebraic varieties where it is expected that singularities will form \cite{FIK, LT, Ts, So, SoT1, SoT2, SoT3, SW1, SW2, T, TZha, Zha1, Zha2}.  One motivation is to understand how the Ricci flow can pass through a singularity and continue on a new manifold.  Such a procedure is sometimes referred to as \emph{canonical surgery}  \cite{H, P1}.  The rigid framework of K\"ahler geometry provides a good model for the study of singularity formation for the general Ricci flow.  Another motivation comes from algebraic geometry and the minimal model program.  It has been conjectured that the K\"ahler-Ricci flow will give a new analytic analogue of the Mori program of classifying algebraic varieties up to birational equivalence \cite{SoT3, T}. 

Indeed, it was conjectured in \cite{SoT3} that the K\"ahler-Ricci flow will either deform a projective algebraic variety $X$ to its minimal model via finitely many divisorial contractions and flips in the Gromov-Hausdorff sense, and then converge (after normalization) to a generalized K\"ahler-Einstein metric on the canonical model of $X$, or collapse in finite time.   The existence of weak solutions of the K\"ahler-Ricci flow through divisorial contractions and flips was proved in \cite{SoT3}.  The Gromov-Hausdorff convergence of the flow is still open in general.

In \cite{SW1}, the authors investigated the Gromov-Hausdorff convergence of the K\"ahler-Ricci flow in the case of 
 Hirzebruch surfaces, and analogous higher dimensional $\mathbb{P}^1$-bundles, with $\textrm{U}(n)$-symmetric initial data.  It was shown that the K\"ahler-Ricci flow will, in the sense of Gromov-Hausdorff,  either contract a divisor, collapse a dimension or shrink to a point.  This confirmed conjectures of Feldman-Ilmanen-Knopf \cite{FIK}. 

 In the prequel to this paper \cite{SW2}, the authors considered the case of the K\"ahler-Ricci flow contracting exceptional divisors (in the usual sense of `blow-down').  More precisely, we assumed that   there exists a surjective holomorphic map $\pi$ from $X$ to a manifold $Y$ blowing down disjoint exceptional divisors $E_1, \ldots, E_k$, which are 
 codimension 1 submanifolds of $X$ biholomorphic to $\mathbb{P}^{n-1}$ and with normal bundle $\mathcal{O}(-1)$.  The map $\pi$ is a blow-down map contracting the $E_i$ to distinct points $y_1, \ldots, y_k$.  In \cite{SW2} we established  that, under a necessary cohomological condition, the K\"ahler-Ricci flow will blow down the exceptional divisors in the Gromov-Hausdorff sense, and smoothly away from the divisors, and continue on the new manifold.   Note that, in this non-collapsing case, the Ricci curvature is not bounded from below along the K\"ahler-Ricci flow \cite{Zha3}.  Thus, since one cannot directly apply  Cheeger-Colding type results (for example \cite{CC}) from Riemannian geometry, we use other methods.
 
In the current paper, we build on results of \cite{SW1, SW2}.   We improve the result of \cite{SW2} by showing that the Gromov-Hausdorff limit on $Y$ coincides with the metric completion of the limiting metric $g_T$ on $Y \setminus \{y_1, \ldots, y_k \}$ (as conjectured in \cite{SW2}).  Moreover, we
 extend the results of \cite{SW2} to  deal with the case of submanifolds $\mathbb{P}^{n-1}$ with normal bundle $\mathcal{O}(-k)$ blowing down to orbifold points.  We then apply these results to  the $\mathbb{P}^1$-bundles considered in \cite{SW1} (without assuming any symmetry of the initial data). We also investigate  minimal surfaces of general type with only distinct irreducible $(-2)$-curves and show that the normalized K\"ahler-Ricci flow converges in the Gromov-Hausdorff sense to a K\"ahler-Einstein orbifold.


We now state our results more precisely. 
Let $(X, \omega_0)$ be a compact K\"ahler manifold of complex dimension $n \ge 2$.  We consider the K\"ahler-Ricci flow,
\begin{equation} \label{krf0}
\frac{\partial}{\partial t} \omega = - \textrm{Ric}(\omega), \quad \omega|_{t=0} = \omega_0.
\end{equation}
As long as the flow exists, the K\"ahler class of $\omega(t)$ is given by
\begin{equation}
[\omega(t)] = [\omega_0] + t  c_1(K_X) >0.
\end{equation}
The first singular time  $T \in (0, \infty]$ is characterized by
\begin{equation} \label{T}
T = \sup \{ t \in \mathbb{R} \ | \ [\omega_0] + t  c_1(K_X) >0 \}.
\end{equation}
Indeed, it was shown by Tian-Zhang \cite{TZha}, extending a result in \cite{Cao}, that there exists a smooth  solution of the K\"ahler-Ricci flow (\ref{krf0}) for $t$ in $[0,T)$.  If $T$ is finite then the flow must develop a singularity as $t \rightarrow T^-$.

We introduce some terminology.  We want to consider a type of \emph{exceptional divisor} which generalizes the one dealt with in \cite{SW2}.
Suppose that $X$ contains a submanifold $E \cong \mathbb{P}^{n-1}$ with normal bundle $\mathcal{O}(-k)$ for some $k\ge 1$.  Suppose that $\pi: X \rightarrow Y$ is the map blowing down $E$, where 
$Y$ is a K\"ahler orbifold of complex dimension $n$ with an orbifold point $y_0$ with a neighborhood corresponding to a neighborhood of  $0 \in \mathbb{C}^n/\mathbb{Z}_k$.  Here $j \in \mathbb{Z}_k$ acts on $\mathbb{C}^n$ by 
\begin{equation}
j \cdot (z^1, \ldots, z^n) = (e^{2\pi \sqrt{-1}j/k} z^1, \ldots, e^{2\pi \sqrt{-1}j/k} z^n).
\end{equation}
  We say that $E$ is the \emph{$(-k)$ exceptional divisor of  $y_0$} and that $y_0$ is a \emph{$\mathbb{Z}_k$-orbifold point} of $Y$.  For more details on this construction, see Section \ref{sectlocalmodel}.

The case  of the usual blow-down of an exceptional divisor to a point, as in \cite{SW2}, coincides with  $k=1$.  Note that we can of course localize the definition so that $X$ may have a number of disjoint exceptional divisors  which may be blown-down by a map $\pi$ to an orbifold $Y$.

Our main theorem is as follows.

\pagebreak[3]
\begin{theorem} \label{mainthm}
Suppose there exists a map $\pi: X \rightarrow Y$ blowing down the  $(-k_i)$ exceptional divisors $E_{i}$ of the $\mathbb{Z}_{k_i}$-orbifold points $y_{i} \in Y$ for $i=1, \ldots, p$.  Here $Y$ is a compact K\"ahler orbifold  whose only orbifold points are  $y_{1}, \ldots, y_p$.  We assume that the $E_1, \ldots, E_p$ are all disjoint and
\begin{equation} \label{condition}
[\omega_0] + T c_1(K_X) = [\pi^*\omega_{\emph{orb}}],
\end{equation}
for $\omega_{\emph{orb}}$ a smooth orbifold K\"ahler metric on $Y$.    Then:
\begin{enumerate}
\item[(i)]  As $t \rightarrow T^-$,   the metrics $g(t)$ converge to a smooth K\"ahler metric $g_T$ on $X \setminus \bigcup_{i=1}^p E_i$ in $C^{\infty}$ on compact subsets of $X \setminus \bigcup_{i=1}^p E_i.$
Using the map $\pi$ we may also regard $g_T$ as a K\"ahler metric on $Y':= Y \setminus \{ y_1, \ldots, y_p \}$.
\item[(ii)]  Let $d_{g_T}$ be the distance function on $Y'$  given by $g_T$.  Then there exists a unique metric $d_T$ on $Y$ extending $d_{g_T}$ such that $(Y, d_T)$ is a compact metric space homeomorphic to the orbifold $Y$ and  $(Y, d_T)$ is  the metric completion of $(Y', d_{g_T})$.
\item[(iii)] $(X, g(t))$ converges to $(Y, d_T)$ in the Gromov-Hausdorff sense as $t\rightarrow T^-$.
\item[(iv)] There exists a  smooth maximal solution $g(t)$, in the orbifold sense, of the K\"ahler-Ricci flow on $Y$ for $t\in (T, T_Y)$, with $T< T_Y \le \infty$, such that $g(t)$ converges to $g_T$ as $t  \rightarrow T^+$ in $C^{\infty}$ on compact subsets of $Y'$.  Moreover $g(t)$ is uniquely determined by $g_0$.
\item[(v)] $(Y, g(t))$ converges to $(Y, d_T)$ in the Gromov-Hausdorff sense as $t\rightarrow T^+$.
\end{enumerate}
\end{theorem}

Clarifying the uniqueness statement in (iv): we obtain uniqueness at the level of potential functions, in the same sense as in \cite{SoT3, SW2}.  For details, see Section \ref{sectiv} and in particular Lemma \ref{lphi}.

In the case when the  $k_j$ are all equal to 1, the setting of the above theorem coincides with that of \cite{SW2}.  The improvement we make here is the addition of the statement in (ii) that $(Y, d_T)$ is the metric completion of $(Y', d_{g_T})$.  

 In \cite{SW2}, we called the behavior of the K\"ahler-Ricci flow as proved there a  \emph{canonical surgical contraction}.  Modifying this terminology slightly, we will say that the K\"ahler-Ricci flow performs  a canonical surgical contraction if (i)-(v) of Theorem \ref{mainthm} hold.  We also remark that by the Adjunction Formula, (\ref{condition}) can hold only if $k_j <n$.  So in particular if $X$ is a K\"ahler surface then we are necessarily in the situation of \cite{SW2}.

In Section \ref{sectlocalmodel} we give more explanation of the condition (\ref{condition}).  We remark here that although $\pi^*\oorb$ is not a smooth form on $X$ in general, the cohomology class $[\pi^* \oorb]$ is well-defined using the cohomology of currents.  Moreover, the cohomology class $[\oorb]$ on $Y$ contains a smooth nonnegative closed $(1,1)$ form $\omega_{Y}$ which is positive definite away from the orbifold points and pulls-back to a smooth $(1,1)$-form on $X$.

We apply Theorem \ref{mainthm} to  the case of the family of $n$-folds $M_{n,k}$, considered by Calabi \cite{C}, which generalize the Hirzebruch surfaces.  $M_{n,k}$ is a compactification of the blow up of a $\mathbb{Z}_k$-orbifold point, as discussed above, and is a $\mathbb{P}^1$-bundle over $\mathbb{P}^{n-1}$.   We describe the construction now in more detail.
We define $M_{n,k}$ to be the $\mathbb{P}^1$-bundle
\begin{equation}
M_{n,k} = \mathbb{P} (\mathcal{O}(-k) \oplus \mathcal{O}) \label{Mnk}
\end{equation}
over $\mathbb{P}^{n-1}$.  We will assume in this paper that $k \ge 1$ and $n \ge 2$. The case $k=0$ corresponds to the product manifold $\mathbb{P}^1 \times \mathbb{P}^{n-1}$ which will not be dealt with here.    Denote by $D_0$ and $D_{\infty}$ the divisors in $M_{n,k}$ corresponding to sections  of $\mathcal{O}(-k) \oplus \mathcal{O}$ with zero $\mathcal{O}(-k)$ and $\mathcal{O}$ component respectively (see Section \ref{sectbundles}).   $D_0$ is an exceptional divisor with normal bundle $\mathcal{O}(-k)$ of the type discussed above.

There is a map $\pi$ from $M_{n,k}$ to an orbifold $Y_{n,k}$ which blows down the exceptional divisor $D_0$.  We describe this in detail in Section \ref{sectbundles}.
The orbifold $Y_{n,k}$ is the weighted projective space
\begin{align}
Y_{n,k} = \{ (Z_0, \ldots, Z_n) \in \mathbb{C}^{n+1} \}/ \sim,
\end{align}
where $(Z'_0, \ldots, Z'_n) \sim (Z_0, \ldots, Z_n)$ if there exists $\lambda \in \mathbb{C}^*$ such that
\begin{align}
(Z'_0, Z'_1, \ldots, Z'_n) = (\lambda^k Z_0, \lambda Z_1, \ldots, \lambda Z_n).
\end{align}
We write elements of $Y_{n,k}$ as $[Z_0, \ldots, Z_n]$.  Then $Y_{n,k}$ has a single $\mathbb{Z}_k$-orbifold point at $[1,0,\ldots, 0]$.  The map $\pi$ restricted to $M_{n,k} \setminus D_0$ is a biholomorphism onto $Y_{n,k} \setminus \{ [1,0, \ldots, 0 ]\}$ and $\pi(D_0) = [1,0, \ldots, 0]$.

All of the manifolds $M_{n,k}$ admit K\"ahler metrics.  Indeed, the cohomology classes of the line bundles $[D_0]$ and $[D_{\infty}]$ span $H^{1,1}(M_{n,k}; \mathbb{R})$ and every K\"ahler class $\alpha$ can be  written uniquely as
\begin{equation} \label{alpha}
\alpha = \frac{b}{k} [D_{\infty}] - \frac{a}{k} [D_0]
\end{equation}
for constants $a$, $b$ with $0< a < b$. 
We consider the special case when $1 \le k \le n-1$ and the initial K\"ahler metric $\omega_0$ lies in the class
\begin{equation}
\alpha_0 =  \frac{b_0}{k} [D_{\infty}] - \frac{a_0}{k} [D_0], \quad \textrm{with} \quad 
(n+k) a_0 < (n-k) b_0.
\end{equation}
In this case  it is known (see \cite{SW1} and the references therein) that a solution to (\ref{krf0}) on $M_{n,k}$ starting at $\omega_0 \in \alpha_0$ exists on $[0,T)$ with $T=a_0/(n-k)$.  As $t \rightarrow T$, the coefficient $a_t$ tends to zero.  Moreover, the metrics $g(t)$ converge smoothly on compact subsets of $M_{n,k} \setminus D_0$ to a K\"ahler metric $g_T$. It was shown in \cite{SW1} that the K\"ahler-Ricci flow contracts the divisor $D_0$ in the sense of Gromov-Hausdorff 
under the assumption that the initial metric satisfies a $\textrm{U}(n)$ symmetry.   This result was conjectured by Feldman-Ilmanen-Knopf in their detailed analysis of self-similar solutions of the K\"ahler-Ricci flow \cite{FIK}.  In the following theorem we make no symmetry assumption.

\pagebreak[3]
\begin{theorem} \label{thmhirz}  On $M_{n,k}$,
let $\omega(t)$ be a solution of the K\"ahler-Ricci flow (\ref{krf0}) for $t \in [0,T)$.  Assume that  the initial K\"ahler metric $\omega_0$ lies in the K\"ahler class $\alpha_0$ given by $a_0, b_0$ satisfying $0< a_0< b_0$.  Assume
\begin{enumerate}
\item[(a)] $1 \le k \le n-1$;  \ and
\item[(b)] $(n+k) a_0 < (n-k) b_0$.
\end{enumerate}
Then the K\"ahler-Ricci flow $g(t)$ on  $M_{n,k}$ performs a canonical surgical contraction with respect to $\pi: M_{n,k} \rightarrow Y_{n,k}$, the divisor $D_0$ and $y_0 = [1,0,\ldots, 0] \in Y_{n.k}$ in the sense of Theorem \ref{mainthm}.   
\end{theorem}

If $n=2$ the result of Theorem \ref{thmhirz}  is  contained in \cite{SW2}, except for the assertion about the metric completion.  In that case $k=1$ and the manifold $M_{2,1}$ coincides with $\mathbb{P}^2$ blown-up at one point, $D_0$ with the exceptional curve. 


Next we consider the case
of the  K\"ahler-Ricci flow on a minimal surface of general type.  Recall that that a projective surface $X$ of general type (Kodaira dimension 2)
 is called a \emph{minimal surface} if  $X$ contains no curves $C$ with $C\cdot C = -1$.  It is well-known that every surface of general type is birational to a minimal surface (its  `minimal model').

For $m$ sufficiently large, the holomorphic sections of $K_X^m$ induce a holomorphic map $\Phi: X \rightarrow \mathbb{P}^N$  for some $N$.  The image of $\Phi$ is the  \emph{canonical model} $\Xc$.  The map $\Phi$ contracts the $(-2)$-curves on $X$.  The canonical model $\Xc$ may not be a smooth projective variety in general, but it is at worst  an orbifold with finitely many orbifold singularities.  Moreover,  $\Xc$ has negative first Chern class and admits an orbifold K\"ahler-Einstein metric $\omega_{\textrm{KE}}$ (see \cite{Y1, A, Kob}).    

For convenience we consider the \emph{normalized} K\"ahler-Ricci flow on $X$: 
\begin{equation} \label{krf1}
\frac{\partial}{\partial t} \omega = - \textrm{Ric}(\omega) - \omega, \quad \omega|_{t=0} = \omega_0,
\end{equation}
for $\omega_0$ a smooth K\"ahler metric on $X$.  Denote by $C_1, \ldots, C_k$ the $(-2)$-curves on $X$.  It was shown by Tsuji \cite{Ts} and Tian-Zhang \cite{TZha} that a solution to (\ref{krf1}) exists for all time and converges smoothly on compact subsets of $X \setminus \bigcup_{i=1, \ldots k} C_i$ to $\Phi^* \omega_{\textrm{KE}}$.  We make the assumption that $X$ contains only distinct irreducible $(-2)$ curves (so, in particular, not intersecting), and we prove the following:

\begin{theorem} \label{thmmin}  Let $X$ be a minimal surface of general type.  Assume  that $X$ contains only distinct irreducible $(-2)$-curves.  Let $\omega(t)$ be a solution of (\ref{krf1}) for $t$ in $[0, \infty)$.  Then $(X, \omega(t))$ converges in the Gromov-Hausdorff sense to $(X_{\emph{can}}, \omega_{\emph{KE}})$ as $t \rightarrow \infty$.  
\end{theorem}

To prove the result in this paper  we make extensive use of methods from \cite{SW2}.  To avoid repetition we focus here only on the new arguments that are needed in these cases.

Finally, we remark that the authors were informed by G. Tian that he and G. La Nave are currently writing up some results on the V-soliton equation, introduced in \cite{LT}, which may treat some of the cases considered  in \cite{SW2}.

\pagebreak[3]
\section{The local model} \label{sectlocalmodel}

In this section we describe the local model of a $(-k)$ exceptional divisor blowing down to an orbifold point.  Most of this material is  well-known, but for the convenience of the reader we collect together here some useful facts which will be needed later.

Let $L$ be the $\mathcal{O}(-k)$ line bundle over $\mathbb{P}^{n-1}$, for $k \ge 1$.  We give a description of the total space of $L$ as follows.   Writing $[Z_1, \ldots, Z_n]$ for the homogeneous coordinates on $\mathbb{P}^{n-1}$, we   define 
\begin{equation}
L = \{ ([Z_1, \ldots, Z_n], \sigma) \in \mathbb{P}^{n-1} \times \mathbb{C}^n \ | \ \sigma \textrm{ is in the line  } \lambda \mapsto (\lambda(Z_1)^k, \ldots, \lambda(Z_n)^k) \},
\end{equation}
and let  $p: L \rightarrow \mathbb{P}^{n-1}$  be the projection onto the first factor.  Each fiber  $p^{-1}([Z_1, \ldots, Z_n])$ is a line in $\mathbb{C}^n$.  $L$ can be given  $n$ complex coordinate charts
\begin{equation}
U_i = \{ ([Z_1, \ldots, Z_n], \sigma) \in L \ | \ Z_i \neq 0 \}, \quad \textrm{for } i=1, \ldots, n.
\end{equation}
On $U_i$ we have coordinates $w_{(i)}^j$ for $j=1, \ldots, n$ with $j \neq i$ and $y_{(i)}$.  The $w_{(i)}^j$ are defined by
\begin{equation}
w^{j}_{(i)} = Z_j/Z_i, \qquad \textrm{for } j \neq i,
\end{equation}
and $y_{(i)}$ by
\begin{equation}
\sigma = \frac{y_{(i)}}{(Z_i)^k}  ((Z_1)^k, \ldots, (Z_n)^k).
\end{equation}
On $U_i \cap U_{\ell}$ with $i \neq \ell$ we have
\begin{equation}
w_{(i)}^j = \frac{w^j_{(\ell)}}{w_{(\ell)}^i} \textrm{ for } j \neq i,\ell, \quad w_{(i)}^{\ell} = \frac{1}{w^i_{(\ell)}} \quad \textrm{and} \quad y_{(i)} = y_{(\ell)} \left( \frac{Z_{i}}{Z_{\ell}}   \right)^k = y_{(\ell)} (w^i_{(\ell)})^k.
\end{equation}

Now let $E$ be the submanifold of $L$ defined by the zero section of $L$ over $\mathbb{P}^{n-1}$.  Denote by $[E]$ the pull-back line bundle $p^*L$ over $L$, which corresponds to the hypersurface $E$.  Writing the transition functions of $[E]$  in $U_i \cap U_{\ell}$ as 
$
t_{i \ell} = \left( \frac{Z_i}{Z_{\ell}} \right)^k = \frac{y_{(i)}}{y_{(\ell)}}$,
 we have a section $s$ over $[E]$ given by
\begin{equation} \label{se}
s_i: U_i \rightarrow \mathbb{C}, \qquad s_i = y_{(i)}.
\end{equation}
We can define a Hermitian metric $h$ on the fibers of $[E]$ by
\begin{equation} \label{he}
h_i = \frac{\left( \sum_{j=1}^n |Z_j|^{2}\right)^k}{|Z_i|^{2k}} \quad \textrm{on} \quad U_i.
\end{equation}
Namely, $h$ is the pull-back of $h_{\textrm{FS}}^{-k}$ where $h_{\textrm{FS}}$ is the Fubini-Study metric on $\mathcal{O}(1)$.
We have
\begin{equation} \label{s}
|s|^2_h = |y_{(i)}|^2 \frac{\left(\sum_{j=1}^n |Z_j|^{2}\right)^k}{|Z_i|^{2k}} \quad \textrm{on} \quad U_i.
\end{equation}

Note that there is a map $\pi: L \rightarrow \mathbb{C}^n/\mathbb{Z}_k$ given by
\begin{equation}
\pi( [Z_1, \ldots, Z_n], \sigma) = (\lambda Z_1, \ldots, \lambda Z_n) \in \mathbb{C}^n/\mathbb{Z}_k,
\end{equation}
where $\lambda$ is a complex number  satisfying $\lambda^k (Z_1^k, \ldots, Z_n^k) = \sigma$.  The map $\pi$ is a biholomorphism away from $E$ and $\pi^{-1}(0)=E$.

The map $\pi$ is not a smooth map $L \rightarrow \mathbb{C}^n/\mathbb{Z}_k$ in the orbifold sense.  In general, a smooth orbifold function $f$ on $\mathbb{C}^n/\mathbb{Z}_k$ (that is, an $f$ which lifts to a smooth $\mathbb{Z}_k$-invariant function $\tilde{f}$ on $\mathbb{C}^n$) may not pull-back via $\pi$ to a smooth function on $L$.  However any smooth orbifold function $f$ on $\mathbb{C}^n/\mathbb{Z}_k$ whose lift $\tilde{f}: \mathbb{C}^n \rightarrow \mathbb{R}$ is of the form $\tilde{f}(z) = \mu(r^{2k})$, where $r^2 = |z^1|^2+ \cdots + |z^n|^2$ and $\mu: [0, \infty) \rightarrow \mathbb{R}$ is a smooth function, has the property that $\pi^*f$ is smooth on $L$.  In particular, note that, by (\ref{s}), the pull-back via $\pi$ of $r^{2k}$ is the smooth function
\begin{equation} \label{r2k}
\pi^* r^{2k} = |s|^2_h.
\end{equation}
on $L$.

Let $\oeuc$ be the standard orbifold metric on $\mathbb{C}^n/\mathbb{Z}_k$, which lifts to the Euclidean metric on $\mathbb{C}^n$.  Using coordinates $z^i$ on $\mathbb{C}^n$ we write $\oeuc$ as 
\begin{equation} \label{oeucdef}
\oeuc = \frac{\sqrt{-1}}{2\pi} \sum_i dz^i \wedge d\ov{z^i}.
\end{equation}
The metric $\oeuc$ does not pull-back via $\pi$ to a smooth metric on $L$.  However we can consider the nonnegative orbifold $(1,1)$-form
\begin{equation} \label{omegainfty}
\osm = \ddbar( |\underline{z} |^{2k}),
\end{equation}
on $\mathbb{C}^n/\mathbb{Z}_k$, where 
$|\underline{z}|^2 := r^2 := |z^1|^2 + \cdots  + |z^n|^2.$  Note that $\osm$  is positive definite away from $0$.
By the discussion above $\pi^*\osm$ is a smooth closed $(1,1)$-form on $L$.

\begin{lemma} \label{lemmakahler}
If $c>0$ is any constant then
\begin{equation} \label{claim}
\omega := \pi^* \osm - c R(h) \qquad \textrm{is K\"ahler on } L,
\end{equation}
where $R(h)$ is the curvature of the Hermitian metric $h$ on $[E]$ as described above. 
\end{lemma}
\begin{proof} In $U_i$ we have
\begin{equation}
-R(h) = k \ddbar \log (1+  \sum_{j \neq i} |w^j_{(i)}|^{2}) \ge 0.
\end{equation}
Then since $\pi^*\omega_{\infty}$ is positive definite on $L \setminus E$,
we only need to check that $\omega$ is positive definite at points of  $E$. For simplicity of notation, we compute in the chart $U_n$, and write $\zeta_j = w^j_{(n)}$ for $j=1, \ldots, n-1$ and $\zeta_n = y_{(n)}$.  Then
\begin{equation}
\omega = \ddbar \left( |\zeta_n |^{2}  (1+ \sum_{j=1}^{n-1} |\zeta_j|^{2})^k \right) + c k \ddbar \log (1+  \sum_{j= 1}^{n-1} |\zeta_j|^{2}),
\end{equation}
and at a point of $E$ we have $\zeta_n=0$ and hence
\begin{equation}
\omega \ge \frac{\sqrt{-1}}{2\pi}  d \zeta_n \wedge d \ov{\zeta_n} +  c k \ddbar \log (1+  \sum_{j= 1}^{n-1} |\zeta_j|^{2}),
\end{equation}
which is clearly positive definite.\qed
\end{proof}

It will be useful to write down this reference metric in local coordinates on the orbifold $\mathbb{C}^n/\mathbb{Z}_k$.  Abusing notation somewhat, we write $\omega$ for the metric $(\pi|_{L\setminus E}^{-1})^*\omega$ on $(\mathbb{C}^n/\mathbb{Z}_k) \setminus \{0 \}$.  We write $B$ for the unit ball in $\mathbb{C}^n/\mathbb{Z}_k$.

\begin{lemma}  \label{lemmahato} For $\omega$ as defined in Lemma \ref{lemmakahler} we have, on $(\mathbb{C}^n/\mathbb{Z}_k) \setminus \{ 0 \}$,
\begin{equation} \label{hatg}
\omega = \frac{\sqrt{-1}}{2\pi} \sum_{i,j} \left( k r^{2(k-1)} (\delta_{ij} + \frac{(k-1)}{r^2} \ov{z^i}{z^j})  + \frac{ck}{r^2}( \delta_{ij} - \frac{\ov{z^i} z^j}{r^2} ) \right) dz^i \wedge d\ov{z^j},
\end{equation}
and, in particular, on $B \setminus \{0 \}$, 
\begin{equation}\label{orbineq}
k r^{2(k-1)} \omega_{\emph{Eucl}} \le \omega \le \frac{C}{r^2} \omega_{\emph{Eucl}},
\end{equation}
for $C = 2k-1+ck$.
\end{lemma}
\begin{proof}
From the definition of $\omega$ and (\ref{oeucdef}) we have, away from $0$,
\begin{equation}
\omega = \ddbar (|\underline{z}|^{2k}) + ck \ddbar \log |\underline{z}|^2,
\end{equation}
and the lemma follows from a routine computation. \qed
\end{proof}

Note that we can rewrite (\ref{orbineq}) on $\pi^{-1}(B \setminus \{ 0 \})$ as follows:
\begin{equation} \label{orbineqs}
k |s|_h^{2(k-1)/k} \pi^* \oeuc \le \omega \le \frac{C}{|s|_h^{2/k}} \pi^* \oeuc.
\end{equation}

Next we discuss the condition (\ref{condition}) in the statement of Theorem \ref{mainthm}.   We have the following lemma.
 
\begin{lemma} \label{lemmaoy}
Under the conditions of Theorem \ref{mainthm}, there exists a smooth orbifold function $f$ on $Y$ such that
\begin{enumerate}
\item[(i)] $\omega_Y : = \omega_{\emph{orb}} + \ddbar f$
is a smooth closed nonnegative orbifold $(1,1)$ form on $Y$ which is positive away the orbifold points $y_1, \ldots, y_p$;
\item[(ii)]  $\pi^* \omega_Y$ is a smooth closed $(1,1)$-form on $X$.
\item[(iii)]  We can choose an Hermitian metric $h$ on the line bundle associated to the divisor $E_1 + \cdots + E_p$ so that for $c>0$ sufficiently small,
\begin{equation}
\omega_X := \pi^*\omega_{Y} - c R(h)
\end{equation}
is K\"ahler on $X$.
\end{enumerate}
\end{lemma}
\begin{proof}
Assume for simplicity that $\pi$ blows down only a single $(-k)$ exceptional divisor $E$ to a point $y \in Y$.  By the results of \cite{SW2} we may suppose $k \ge 2$.   We can choose coordinates so that, working  in an orbifold uniformization chart  of the orbifold point $y$, which we identify with an open set $U$ in $\mathbb{C}^n$, 
\begin{equation}
\oorb = \ddbar \varphi, \ \textrm{with} \  \varphi = |\underline{z}|^2 + \textrm{O}(|\underline{z}|^3).
\end{equation}
For a small constant $\delta>0$ we choose a smooth radially symmetric cut-off function $\psi$  such that $\psi \equiv 1$ on $B_{\delta/2}$ and $\psi \equiv 0$ outside $B_{\delta}$, where $B_R$ denotes the ball of radius $R$ in $\mathbb{C}^n$.   Extend $\psi$ to be a smooth function  $Y$ which vanishes outside $B_{\delta}$.
 We may assume that $ \| \psi \| \le 1$,  $\| \psi\|_{C^1(Y)} \le C/\delta$ and $\| \psi \|_{C^2(Y)}  \le C/\delta^2$ for some uniform constant $C$.  Define for $\varepsilon>0$ to be determined later,
\begin{equation}
\omega_{\textrm{orb}, \ve} = \oorb + \ddbar (\psi \cdot (\ve + |\underline{z}|^{2k})^{1/k} - \psi \cdot \varphi).
\end{equation}
Since $\pi^*(\ve + |\underline{z}|^{2k})^{1/k}$ is a smooth function on $X$, we see that $\pi^* \omega_{\textrm{orb}, \ve}$ is a smooth $(1,1)$ form on $X$.  Moreover, on $B_{\delta/2}$ we have
\begin{equation}
\omega_{\textrm{orb}, \ve} =  \ddbar  (\ve + |\underline{z}|^{2k})^{1/k},
\end{equation}
which is nonnegative on $B_{\delta/2}$ and positive on $B_{\delta/2} \setminus \{ 0 \}$.  
Compute
\begin{align} \nonumber
\partial_i \partial_{\ov{j}} \left( ( \ve + |\underline{z}|^{2k})^{1/k}  - |\underline{z}|^2 \right)  = & \ (\ve + |\underline{z}|^{2k})^{(1-2k)/k} |\underline{z}|^{2(k-1)} \left( \ve (\delta_{ij} + (k-1) \frac{\ov{z^i} z^j}{|\underline{z}|^2}) \right. \\
 & \left.   + \   |\underline{z}|^{2k} \delta_{ij} - (\ve + |\underline{z}|^{2k})^{(2k-1)/k} |\underline{z}|^{-2(k-1)} \delta_{ij} \right), \\
 \partial_{\ov{j}} \left( (\ve + |\un{z}|^{2k})^{1/k} - |\un{z}|^2 \right) = & \ z_j \left( (\ve + |\un{z}|^{2k})^{(1-k)/k} | \un{z}|^{2(k-1)} -1 \right)
\end{align}
We choose $\ve = \delta^{2k}/A$ for some large constant $A$.  Then in $B_{\delta} \setminus B_{\delta/2}$ we have
\begin{equation}
\frac{|\un{z}|^{2k}}{A} \le \ve \le \frac{ |\un{z}|^{2k}}{2^{-2k}A},
\end{equation}
and a straightforward computation gives, in $B_{\delta} \setminus B_{\delta/2}$,
\begin{align}
\left| \partial_i \partial_{\ov{j}} \left( ( \ve + |\underline{z}|^{2k})^{1/k}  - \varphi \right) \right| & \le \frac{1}{C'} \\
| \partial_{\ov{j}} \left( (\ve + |\un{z}|^{2k})^{1/k} - \varphi \right) | & \le \frac{\delta}{C'} \\
| (\ve + |\un{z}|^{2k})^{1/k} - \varphi | & \le \frac{\delta^2}{C'},
\end{align}
where we can make $C'$ as large as we like by choosing $\delta$ sufficiently small and $A$ sufficiently large.

Thus we can make the norm
\begin{equation}
\| \ddbar (\psi \cdot (\ve + |\underline{z}|^{2k})^{1/k} - \psi \cdot \varphi) \|_{C^0(Y\setminus B_{\delta/2})}
\end{equation}
small enough so that $\omega_{\textrm{orb}, \ve}$ is positive definite on $Y \setminus B_{\delta/2}$.    Hence $\omega_{\textrm{orb}, \ve}$ is a smooth closed nonnegative $(1,1)$ form on $Y$ which is positive away from $0$ and which pulls back via $\pi$ to a smooth closed $(1,1)$ form on $X$.  To ensure that we have (iii) we make one final modification.

For a small constant $\eta>0$, define 
\begin{equation}
f = \psi \cdot (\ve + |\underline{z}|^{2k})^{1/k} - \psi \cdot \varphi + \eta \, \psi \cdot |\un{z}|^{2k}
\end{equation}
and
\begin{equation}
\omega_Y = \oorb + \ddbar f = \omega_{\textrm{orb}, \ve} + \ddbar ( \eta \, \psi \cdot |\un{z}|^{2k}).
\end{equation}
Then if $\eta$ is sufficiently small we see that $\omega_Y$ satisfies (i) and (ii).  Moreover, we see that $\omega_Y \ge \eta \ddbar |\un{z}|^{2k}$ in a neighborhood of $y$.  Applying Lemma \ref{lemmakahler} we can choose a Hermitian metric $h$ on  $[E]$ in a neighborhood of $E$ so that
$\pi^*\omega_Y - c R(h)$ is K\"ahler in that neighborhood  for any $c>0$.  A simple patching argument gives (iii).
\qed 
\end{proof}

Observe then that $\pi^*\oorb = \pi^* \omega_Y - \ddbar (\pi^* f)$ is a closed positive (1,1) current and thus the cohomology class  $[\pi^* \oorb]$ is well-defined and equal to 
 $[\pi^* \omega_Y]$.  Hence we can replace Condition (\ref{condition}) in the statement of Theorem \ref{mainthm} by
 \begin{equation}
[\omega_0] + T c_1(X) = [\pi^* \omega_Y], 
 \end{equation}
 where $\omega_Y$ satisfies (i)-(iii) of Lemma \ref{lemmaoy}.

\pagebreak[3]
\section{Estimates away from the exceptional divisors} \label{sectestimates}

In this section we begin the proof of Theorem \ref{mainthm} by recalling the setup and basic estimates as in \cite{SW2}.  First, we rewrite the K\"ahler-Ricci flow as a flow of potentials.  Let $\omega_Y$ be the $(1,1)$-form on $Y$ constructed in Section \ref{sectlocalmodel}.
Define reference $(1,1)$-forms $\hat{\omega}_t$ on $X$ for $t\in [0,T]$ by
\begin{equation} \label{referencemetric}
\hat{\omega}_t = \frac{1}{T} ((T-t) \omega_0 + t\pi^* \omega_Y).
\end{equation}
Then for $t \in [0,T)$, $\hat{\omega}_t$ is a K\"ahler form in the cohomology class $[\omega(t)]$.
Let $\Omega$ be the unique volume form on $X$ with $\ddbar \log \Omega = \ddt{} \hat{\omega}_t$ and $\int_X \Omega=1$.  Let $\varphi=\varphi(t)$ be the solution of
\begin{equation}
\ddt{\varphi} = \log \frac{ (\hat{\omega}_t + \ddbar \varphi)^n}{\Omega}, \quad \varphi|_{t=0} =0.
\end{equation}
Then $\omega = \hat{\omega}_t + \ddbar \varphi$ solves the K\"ahler-Ricci flow (\ref{krf0}).  Moreover:

\begin{lemma} \label{basicestimates} We have
\begin{enumerate}
\item[(i)]  There exists a uniform $C>0$ such that $\| \varphi\|_{L^{\infty}} + \dot{\varphi} \le C$.
\item[(ii)]  There exists a uniform $c>0$ such that $\omega(t) \ge c \pi^*\omega_Y$.
\item[(iii)]  For every compact $K \subset X \setminus \bigcup_{i=1}^p E_i$ there exist constants $C_{K,j}$ for $j=0,1,2, \ldots$ such that
$\| \omega(t) \|_{C^j(K)} \le C_{K,j}$. 
\item[(iv)] As $t \rightarrow T^-$, $\varphi(t)$ converges pointwise on $X$ to a bounded function $\varphi_T$ satisfying 
\begin{equation}
\omega_T:= \pi^* \omega_Y + \ddbar \varphi_T \ge 0,
\end{equation}
which is a smooth K\"ahler form outside $E_1, \ldots, E_p$.
\item[(v)] $\omega(t)$ converges weakly in the sense of currents, and in $C^{\infty}$ on compact subsets of $X \setminus  \bigcup_{i=1}^p E_i$,  to the closed positive $(1,1)$-current $\omega_T$.
\end{enumerate}
\end{lemma}
\begin{proof}  This lemma is essentially due to Tian-Zhang \cite{TZha} (see also \cite{SoT1, Zha1} and the expositions in \cite{SW1, SW2}).  The only minor difference here occurs in (ii) since $\pi$ is no longer a holomorphic map between manifolds.  However, the same maximum principle argument for  bounding $\log \tr{\omega(t)}{ \pi^* \omega_Y}$ from above can be applied  since $\pi$ is holomorphic away from the exceptional divisors and 
 this quantity tends to negative infinity along $E_1, \ldots, E_p$.   \qed
\end{proof}

Part (i) of Theorem \ref{mainthm} then follows immediately from the lemma.  To establish parts (ii)-(v) we need estimates near the exceptional divisors $E_i$.

\pagebreak[3]
\section{Estimates near the exceptional divisor} \label{secthirz}

In this section we establish some estimates needed for parts (ii) and (iii) of Theorem \ref{mainthm}.  For simplicity we assume that there is only one exceptional divisor $E$ which blows down via $\pi$ to the orbifold point  $y_0 \in Y$.  To obtain estimates near $E$, we will work in a local uniformizing chart around the orbifold point, which we identify with the unit ball $B$ in $\mathbb{C}^n$.  As in Section \ref{sectlocalmodel}, we write $B_R$ for the ball in $\mathbb{C}^n$ of radius $R>0$ and $\oeuc$ for the Euclidean metric, which is uniformly equivalent in $B$ to $\oorb$.  We write $\omega(t)$ and $\omega_0$ for their images in $B \setminus \{ 0 \}$ under the map $\pi$ (which is a biholomorphism from $\pi^{-1} (B \setminus \{0 \})$ to $B \setminus \{0 \}$).  We use coordinates $z^1, \ldots, z^n$ on $\mathbb{C}^n$ and write $r = \sqrt{|z^1|^2 + \cdots + | z^n |^2}$.

The key estimates are contained in the following lemma (cf. Lemmas 2.5, 2.6 of \cite{SW2}).

\begin{lemma} \label{lemmaest2}
Let $\omega=\omega(t)$ be a solution of the K\"ahler-Ricci flow (\ref{krf0}) on $X$, for $t \in [0,T)$, under the assumptions of Theorem \ref{mainthm}.
Then there exist uniform constants $C$, $\delta'>0$ and $R_0=R_0(n,k) \in (0,1)$ such that on $B_{R_0} \setminus \{0 \}$,
\begin{enumerate}
\item[(i)] $\displaystyle{\omega \le \frac{C}{r^2} \, \omega_{\emph{Eucl}}.}$
\item[(ii)]$\displaystyle{\omega \le \frac{C}{r^{2(1-\delta')}} (\omega_0 + \omega_{\emph{Eucl}}).}$
\item[(iii)] $\displaystyle{
|W|^2_{g} \le \frac{C}{r^{2(1-\delta)}}}$ for $\delta = k/(k+1)>0$ and
where $W = \sum_{i=1}^n \left( \frac{x_i}{r} \frac{\partial}{\partial x_i} + \frac{y_i}{r} \frac{\partial}{\partial y_i} \right)$, the unit length radial vector field with respect to $g_{\emph{Eucl}}$, with $z_i = x_i + \sqrt{-1} y_i$.
\end{enumerate}
\end{lemma}
\begin{proof}  
First, we recall a well-known local computation that holds for any solution $\omega(t)$ of the K\"ahler-Ricci flow (\ref{krf0}), where $\tilde{\omega}$ is a fixed K\"ahler metric on a manifold $X$.  
From \cite{Y1, A, Cao},
\begin{eqnarray} \nonumber
\left( \ddt{} - \Delta \right) \log \tr{\tilde{\omega}}{\omega} & = & \frac{1}{\tr{\tilde{\omega}}{\omega}  } \left( - g^{i \ov{j}} \tilde{R}_{i \ov{j}}^{ \ \ k \ov{\ell}} g_{k \ov{\ell}} - g^{i \ov{j}} \tilde{g}^{k \ov{\ell}} g^{p\ov{q}} \tilde{\nabla}_i g_{k \ov{q}} \tilde{\nabla}_{\ov{j}} g_{p\ov{\ell}} + \frac{ |\nabla \tr{\tilde{\omega}}{\omega}|^2}{\tr{\tilde{\omega}}{\omega}} \right) \\
& \le &   \frac{1}{\tr{\tilde{\omega}}{\omega}  } \left( - g^{i \ov{j}} \tilde{R}_{i \ov{j}}^{ \ \ k \ov{\ell}} g_{k \ov{\ell}}\right) \le  \tilde{C} \, \tr{\omega}{\tilde{\omega}}, \label{tr1}
\end{eqnarray}
for $\tilde{C}$ depending only on the lower bound of the bisectional curvature of $\tilde{g}$.

Let $s$ be a holomorphic section of the line bundle $[E]$ vanishing to order 1 along $E$, defined locally by (\ref{se}).   We define a Hermitian metric $h$ on the line bundle $[E]$ by Lemma \ref{lemmaoy}.  Then from (\ref{r2k}) we see that
 on $B_{R_0}$ for some uniform $R_0 \in (0,1)$,
\begin{equation}
 |s|^{2}_h = r^{2k}.
\end{equation} 
Moreover,   there exists a $c>0$ sufficiently small and a uniform constant $C>0$ such that on $X$,
\begin{equation}
\hat{\omega}_{t} - c R(h) \ge \frac{1}{C} \omega_0,
\end{equation}
where $R(h) = -\frac{\sqrt{-1}}{2\pi} \partial \ov{\partial} \log h$.

Define a function $H_{\ve}$ on $(\ov{B_{R_0}} \setminus \{0 \}) \times [0,T)$ by
\begin{equation}
H_{\ve} = \log \tr{\omega_0}{\omega} + A \log \left( |s|^{2(1+\ve)/k}_h \tr{\oeuc}{\omega} \right) -A^2 \varphi,
\end{equation}
for $\ve>0$, where $A$ is a constant to be determined.   We see from (\ref{orbineqs}) that $H_{\varepsilon}(x,t) \rightarrow -\infty$ as $x$ approaches $0$, for any $t \in [0,T)$.  From Lemma \ref{basicestimates}, we have uniform estimates for $H_{\ve}$ on the boundary of $B_{R_0}$.    Assume there exists a point  $(x_0, t_0) \in ( B_{R_0} \setminus \{0 \} )\times (0,T)$ such that  $H_{\ve}(x_0,t_0) = \sup_{(\ov{B_{R_0}} \setminus \{ 0 \}) \times [0,t_0]} H_{\ve}$.   
At $(x_0,t_0)$, using (\ref{tr1}),
\begin{equation}
0 \le \left( \frac{\partial}{\partial t} - \Delta \right) H_{\ve} \le C' \tr{\omega}{\omega_0} + A\, \tr{\omega}{ \left(A\hat{\omega}_{t_0} - \frac{(1+\ve)}{k} R(h)\right)} - A^2 \log \frac{\omega^n}{\Omega} + A^2n,
\end{equation}
for some uniform $C'>0$.
Choosing $A$ large enough, we obtain
\begin{equation}
A\left(A\hat{\omega}_{t_0} - \frac{(1+\ve)}{k} R(h)\right) \ge (C'+1) \omega_0.
\end{equation}
Thus at $(x_0,t_0)$ we have
\begin{equation}
\left( A^2 \log \frac{\omega^n}{\Omega} + \tr{\omega}{\omega_0} \right) (x_0, t_0) \le C.
\end{equation}
Since, for any $\kappa>0$, the function $\tau \in (0, \infty) \mapsto \log \tau + \kappa/\tau \in (0,\infty)$ is uniformly bounded from below and tends to infinity as $\tau \rightarrow 0^+$, this  implies that
\begin{equation}
(\tr{\omega_0}{\omega}) (x_0,t_0) \le C.
\end{equation}
Hence, at $(x_0,t_0)$, using Lemma \ref{lemmahato},
\begin{equation}
\omega \le C\omega_0 \le \frac{C}{|s|_h^{2/k}} \oeuc.
\end{equation}
Since $\varphi$ is uniformly bounded by Lemma \ref{basicestimates},  it follows that $H_{\ve}$ is uniformly bounded at $(x_0, t_0)$, independent of $\ve$.  Since we already have estimates for $H_{\ve}$ on the boundary of $B_{R_0}$ and at $t=0$ this gives us a uniform bound for $H_{\ve}$ on $(\ov{B_{R_0}} \setminus \{0 \} ) \times [0, T)$.  Letting $\ve \rightarrow 0$ we have (i) since by (\ref{orbineqs}),  $|s|_h^{2/k} \tr{\oeuc}{\omega} \le C\tr{\omega_0}{\omega}$.  Moreover,
\begin{equation}
(\tr{\omega_0}{\omega})^{1/A} (\tr{\oeuc}{\omega}) \le \frac{C}{r^2},
\end{equation}
and thus
\begin{equation}
(\tr{(\omega_0+\oeuc)}{\omega})^{1+1/A} \le \frac{C}{r^2},
\end{equation}
giving (ii).

For (iii), we consider the holomorphic vector field $V= z^i \frac{\partial}{\partial z^i}$ on $B_{R_0} \setminus \{0 \}$.  Using (\ref{hatg}) we see that 
 there exists a uniform $C>0$ such that
\begin{equation} \label{Vest}
\frac{1}{C} r^{2k}\le |V|^2_{\omega_0} \le C r^{2k}.
\end{equation}
We consider the quantity 
\begin{equation}
Q_{\varepsilon} = \log ( |V|_{\omega}^{2(1+\varepsilon)/k} \tr{\oeuc}{\omega}) -t,
\end{equation}
and observe from (\ref{Vest}) that $Q_{\varepsilon}(x)$ tends to negative infinity as $x$ approaches the origin.  By the same computation as in Lemma 2.6 of (\cite{SW2}), if $Q_{\varepsilon}$ achieves a maximum at a point $(x_0,t_0)  \in B_{R_0} \setminus \{0 \} \times (0,T)$ then
\begin{equation}
0 \le \left( \frac{\partial}{\partial t} - \Delta \right) Q_{\varepsilon}(x_0,t_0) \le -1,
\end{equation}
a contradiction.  Hence $Q_{\varepsilon}$ is uniformly bounded from above, and letting $\varepsilon$ tend to zero we obtain 
\begin{equation}
 (\tr{\oeuc}{\omega}) |V|^{2/k}_{\omega} \le C.
 \end{equation} 
 But since $|V|^2_{\omega} \le (\tr{\oeuc}{\omega}) |V|^2_{\oeuc}$ we can multiply by $|V|_{\omega}^{2/k}$ to get
\begin{equation}
|V|^{2(k+1)/k}_{\omega} \le C |V|^2_{\oeuc} \le C r^{2},
\end{equation}
and hence $|V|^2_{\omega} \le C r^{2k/(k+1)}$.  Then
\begin{equation}
|W|_g^2 \le \frac{2}{r^2} |V|_{\omega}^2 \le \frac{C}{r^{2(1- k/(k+1))}},
\end{equation}
giving (iii).
This completes the proof of the lemma.
 \qed
\end{proof}

Note that  the radial bound (iii) of Lemma \ref{lemmaest2} coincides with the estimate of Lemma 2.6 in \cite{SW2} for $k=1$.  We obtain here a stronger bound if $k>1$.  It is now straightforward to establish:

\begin{lemma} \label{klemma} We have:
\begin{enumerate}
\item[(i)]  For $0< r < R_0$, the diameter of the $(2n-1)$-sphere $S_r$ of radius $r$ in $B_{R_0}$ with respect to the metric induced by $g(t)$ is uniformly bounded from above, independent of $r$.
\item[(ii)]  There exists a uniform constant $C$ such that for any $z \in B_{R_0}$ the length of the radial path $\gamma(\lambda)=\lambda z$ for $\lambda \in (0,1]$ with respect to $g(t)$ or $g_T$ is uniformly bounded from above by $C |z|^{\delta}$ for $\delta=k/(k+1)$.
\end{enumerate}
\end{lemma}
\begin{proof} See Lemma 2.7 in \cite{SW2}. \qed
\end{proof}



\pagebreak[3]
\section{Estimates on lengths of paths away from $E$}

In this section we prove some  estimates on lengths of paths with respect to $g(t)$.  These are the key technical estimates needed to identify the metric completion of  $(Y', d_{g_T})$ as in part (ii) of Theorem \ref{mainthm}.
We assume we are in the setting of that theorem.  As in the previous section we will assume that there is only 
one exceptional divisor $E$.

We will often work locally.  Write $B$ for a small ball centered at the origin of $\mathbb{C}^n$, and write $\tilde{B}$ for the corresponding neighborhood in the blow-up:
\begin{equation}
\tilde{B} = \{ ([Z_1, \ldots, Z_n], \sigma) \in \mathbb{P}^{n-1} \times B \ | \ \sigma \textrm{ is in the line } \lambda \mapsto (\lambda Z_1^k, \ldots, \lambda Z_n^k) \}.
\end{equation}
We have a blow-down map $\pi: \tilde{B} \rightarrow B/\mathbb{Z}_k$ sending the exceptional divisor $E \cong \mathbb{P}^{n-1}$ to the origin in $B/\mathbb{Z}_k$.  We will write $g(t)$ for the evolving K\"ahler metric, and we may consider this as a metric  on $\tilde{B} \setminus E$ or via $\pi$ as a metric on on $(B/\mathbb{Z}_k) \setminus \{0 \}$. We will usually regard $g(t)$ as a smooth K\"ahler metric on $B \setminus \{0 \}$ which is invariant under the action of  $\mathbb{Z}_k$.

 Denote by $\oX$ the smooth K\"ahler metric on $\tilde{B}$ which in $\mathbb{C}^n$ is given by (cf. Lemma \ref{lemmahato})
\begin{equation}
\oX =  \frac{\sqrt{-1}}{2\pi} \sum_{i,j} \left( k r^{2(k-1)} ( \delta_{ij} + \frac{(k-1)}{r^2} \ov{z^i} z^j ) + \frac{k}{r^2} (\delta_{ij} - \frac{\ov{z^i} z^j}{r^2} ) \right) dz^i \wedge d\ov{z^j}.
\end{equation}
Since $\omega_X$ is smooth on $\tilde{B}$, it is uniformly equivalent to the fixed initial K\"ahler metric $\omega_0$ restricted to $\tilde{B}$.

We begin with a lemma.  The authors thank Zhou Zhang for suggesting that the estimate (\ref{ote}) below should hold.  Note that this simplifies  the argument of Lemma 3.2 in \cite{SW2} (cf. Lemma \ref{lemmaintN} below).  

\begin{lemma}  \label{oXE} There exists a uniform constant $C$ such that for all $t \in [0,T)$,
\begin{equation} \label{ote}
\omega(t)|_{E} \le C \oX|_{E}.
\end{equation}
\end{lemma}
\begin{proof}
We first claim that on $\tilde{B}$.
\begin{equation} \label{claim00}
\omega \wedge \oX^{n-2} \wedge \omega_{\infty} \le C \oX^{n-1} \wedge \omega_{\infty},
\end{equation}
for some uniform constant $C$, where $\omega_{\infty}$ is given by (\ref{omegainfty}).  To see this, we may compute at a point $z \in \tilde{B} \setminus E$.

From part (i) of Lemma \ref{lemmaest2} and Lemma \ref{lemmahato}, we have
\begin{equation}
\omega \wedge \omega_X^{n-2} \wedge \omega_{\infty}  \le \frac{C}{r^{2(n-1)}} \oeuc^{n-1} \wedge \omega_{\infty}.
\end{equation}
On the other hand,
 since $\oX$, $\oeuc$  and $\omega_{\infty}$ are $\textrm{U}(n)$-invariant we may assume without loss of generality that $z=(z^1, 0, \cdots,0)$.  Then compute at $z$, 
\begin{equation}
\oX = \frac{\sqrt{-1}}{2\pi} \left( k^2 r^{2(k-1)} dz^1 \wedge d\ov{z^1} +    \sum_{j=2}^n (kr^{2(k-1)} + \frac{k}{r^2}) dz^j \wedge d\ov{z^j} \right) \ge \frac{\sqrt{-1}}{2\pi} \frac{k}{r^2} \sum_{j=2}^n dz^j \wedge d\ov{z^j}.
\end{equation}
Hence for some uniform $C>0$,
\begin{equation}
\omega_X^{n-1} \wedge \omega_{\infty} \ge \frac{1}{C r^{2(n-1)}} \oeuc^{n-1} \wedge \omega_{\infty},
\end{equation}
establishing the claim (\ref{claim00}).

To finish the proof of the lemma, fix a point $p \in E$.  We may pick a unitary basis $e_1, \ldots, e_n$ for $T_pX$ with respect to $\oX$ such that the vectors $e_1, \ldots, e_{n-1}$ lie in $T_pE$.  Then $\omega_{\infty}(e_i, \ov{e_j})=0$ 
unless $i=j=n$.  Then at $p$, from (\ref{claim00}),
\begin{equation}
\omega \wedge \oX^{n-2} \wedge \omega_{\infty} (e_1, \ov{e_1}, \ldots, e_n, \overline{e_n}) \le C \oX^{n-1} \wedge \omega_{\infty} (e_1, \ov{e_1}, \ldots, e_n, \overline{e_n}), 
\end{equation}
which implies that
\begin{equation}
\omega \wedge \oX^{n-2}  (e_1, \ov{e_1}, \ldots, e_{n-1}, \overline{e_{n-1}}) \le C \oX^{n-1}  (e_1, \ov{e_1}, \ldots, e_{n-1}, \overline{e_{n-1}}), 
\end{equation}
and hence $\omega \wedge \omega_X^{n-2}|_E \le C \omega_X^{n-1}|_E,$ which gives a uniform upper bound for 
$\tr{\omega_X|_E}{\omega|_E}$ at $p$.  The inequality  (\ref{ote}) follows. \qed
\end{proof}

Next:

\begin{lemma} \label{lemmaintN}
For any projective line $N$ in $E \cong \mathbb{P}^{n-1}$, we have for $t \in [0,T)$,
\begin{equation} \label{intN}
\int_N \omega(t) \le C (T-t),
\end{equation}
for a uniform constant $C$.  Moreover, if $d_{g} =d_{g(t)}$ is the distance function on $X$ associated to the metric $\omega$ then for any points $p,q \in E$ and $t \in [0,T)$,
\begin{equation} \label{dg13}
d_{g}(p,q) \le C(T-t)^{1/3}
\end{equation}
\end{lemma}
\begin{proof}
The inequality (\ref{intN}) follows from the fact that $\omega(t)$ is cohomologous to $\hat{\omega}_t$ given by (\ref{referencemetric}).  Next, (\ref{dg13}) follows from Lemma \ref{oXE} and the argument of Lemma 2.4 in \cite{SSW} (see also Lemma 3.2 in \cite{SW2}). \qed
\end{proof}



Denote by $S_{r}$ the $(2n-1)$-sphere of radius $r>0$, which we assume lies in $B$.  Denote by $A_{r, r'}$ the annulus $A_{r, r'} = \{ z \in B \ | \ r \le |z| \le r' \}$.

Given a metric $g$, we denote by $L_{g}(\gamma)$ the length of a piecewise smooth path $\gamma$ with respect to  $g$. The main result in this section is to bound the lengths of paths in $A_{r, 4r}$ with respect to $g(t)$ and $g_T$.

\begin{theorem} \label{diamA}
There exist uniform constants $c, C$ such that for $r$ with $0<r<c$  the following holds.  
\begin{enumerate}
\item[(i)] For all $t$ with $0 <T-t \le r^{\delta}$ and for all points $z', z''$ in $A_{r, 4r}$ we can find a piecewise smooth path $\gamma=\gamma(t, z', z'')$ in $A_{r,4r}$ from $z'$ to $z''$ such that
\begin{equation} \label{eqn:diamA1}
L_{g(t)} (\gamma) \le C r^{\delta/3}.
\end{equation}
\item[(ii)] For all points $z', z''$ in $A_{r, 4r}$ we can find a piecewise smooth path $\gamma=\gamma(z',z'')$ in $A_{r,4r}$ from $z'$ to $z''$ such that
\begin{equation} \label{eqn:diamA2}
L_{g_T}(\gamma) \le C r^{\delta/3}.
\end{equation}
\end{enumerate}
Here $\delta=k/(k+1)$.
\end{theorem}

As a consequence:

\begin{corollary} \label{cor:diamgT}
For any $\ve>0$ there exists $c_0 >0$ such that
\begin{equation}
\emph{diam}_{g_T} (B_{c} \setminus  \{0\}) < \ve, \quad \textrm{for all } 0< c \le c_0,
\end{equation}
where $B_{c}$ denotes the ball of radius $c$ (with respect to the Euclidean metric) centered at the origin.
\end{corollary}
\begin{proof}[Proof of Corollary \ref{cor:diamgT}]  Let $C\ge 1$ be the maximum of the two constants  of Lemma \ref{klemma}, part (ii), and Theorem \ref{diamA}.  
Choose $c_0>0$ sufficiently small so that $c_0^{\delta/3} <  \frac{\ve}{2C}$.  Take any two points $p,q \in B_{c} \setminus \{ 0 \}$.  Choose $r$ such that $p \in A_{r, 4r} \subset B_c$.  By the radial estimate (part (ii) of Lemma \ref{klemma})  there exists a radial path $\gamma_1$ in $B_c$ from $q'$ to $q$ with $q' \in A_{r,4r}$ and $L_{g_T}(\gamma_1) \le C c^{\delta} \le c^{\delta/3}$.  By Theorem \ref{diamA}, there exists a path $\gamma_2$ in $A_{r,4r}$ from $p$ to $q'$ such that $L(\gamma_2) \le C r^{\delta/3} \le C c^{\delta/3}$.  Let $\gamma$ be the piecewise smooth path in $B_c$ from $p$ to $q$ formed by joining $\gamma_2$ and $\gamma_1$.  Then $L_{g_T}(\gamma) \le 2C c^{\delta/3}  < \ve$.  This proves the required estimate. \qed
\end{proof}

We now begin the proof of Theorem \ref{diamA}.
Since we have the radial estimate for $g(t)$ and $g_T$ (Lemma \ref{klemma}), we may assume without loss of generality that $z', z'' \in S_{2r}$. In fact we can make further assumptions on $z', z''$.  The following lemma will help us.

\begin{lemma} \label{lemma:cxline}
For any $z$ in $S_r$ and any $\theta \in [0,2\pi]$ there exists a smooth path $\gamma_0$ in $S_r$ from $z$ to $ze^{i\theta}$ such that
\begin{equation}
L(\gamma_0) \le C r^{\delta}.
\end{equation}
\end{lemma}
\begin{proof}
Define a  smooth path $\gamma_0: [0,1] \rightarrow S_r$ by $\gamma_0(s) = ze^{is\theta}$.  This is a path from $z$ to $ze^{i\theta}$.  We have $\dds \gamma_0 = i\theta \gamma_0$.  That is, if we write $\gamma_0(s) = (z^1(s), \ldots, z^n(s))$ then 
\begin{equation}
\dds \gamma_0 (s) = i\theta \left( z^1(s) \frac{\partial}{\partial z^1} + \cdots + z^n(s) \frac{\partial}{\partial z^n}  \right)\bigg|_{\gamma_0(s)}= i \theta V_{\gamma_0(s)},
\end{equation}
for $V= z^i \frac{\partial}{\partial z^i}$ as in the proof of Lemma \ref{lemmaest2}, which satisfies $|V|^2_{g(t)} \le C r^{2\delta}$.   Then 
\begin{equation}
L(\gamma_0) = \int_0^1 \sqrt{ \theta^2 | V|^2_{g(t)} (\gamma_0(s))} ds \le C r^{\delta},
\end{equation}
as required. \qed
\end{proof}

The proof of Theorem \ref{diamA} will only require Lemma \ref{lemmaintN} and the local estimates (i) and (iii) of Lemma \ref{lemmaest2} together with their consequence Lemma \ref{klemma}.  
Since $\geu$ and $V$ are invariant under unitary transformations of $\mathbb{C}^n$,  we may, without loss of generality, apply unitary transformations to the points $z'$ and $z''$.  In particular, 
  we may assume that $z'=(2r,0, \ldots, 0)$.  Applying a unitary transformation that fixes the first coordinate, we may assume that $z''$ is of the form $(z^1, z^2, 0, \ldots, 0)$ for some $z^1, z^2 \in \mathbb{C}$ with $|z^1|^2+ |z^2|^2=1$.

Any such point $z'' \in S_{2r}$ must be  of the form 
\begin{equation}
z''= 2r (e^{i \xi_1} \cos \eta, e^{i \xi_2} \sin \eta, 0, \ldots, 0), \quad \textrm{for }  \xi_1, \xi_2 \in [0, 2\pi], \ \eta \in [0,\pi/2].
\end{equation}
  Moreover, we may and do suppose without loss of generality that $ \eta \in [0, \emax]$ for a fixed number $\emax$ with $0<\emax<  \pi/6,$ say (indeed we are free to assume that the projections of $z'$ and $z''$ onto the unit sphere are at most a uniformly small distance from each other with respect to the standard metric).
Applying Lemma \ref{lemma:cxline} we in addition assume that $\xi_1=0$, i.e., 
\begin{equation}
z'= (2r, 0, \ldots, 0), ~~~~z''=2r(\cos\eta, e^{i \xi_2}\sin\eta, 0 \ldots, 0).
\end{equation}
 Next, by  applying a unitary transformation we may assume that
\begin{equation}
z'=2r (\cos \theta, - \sin \theta, 0 \ldots, 0), \quad z'' = 2r (\cos \theta, \sin \theta, 0, \ldots, 0)
\end{equation}
for $\theta = \eta/2 \in [0, \emax/2]$.  Note that from now  we take $\theta$ to be this fixed number in $[0, \emax/2]$.

Next define a complex line $L$ in $\mathbb{C}^n$ by 
\begin{equation}
L = \{ (2r \cos \theta, z, 0 \ldots, 0) \in \mathbb{C}^n \ | \ z \in \mathbb{C} \}.
\end{equation}
Observe that $L$ contains the points $z', z''$.  We also define a cone $\Cone$ in $\mathbb{C}^n$ by
\begin{equation}
\Cone = \{ (\rho e^{i \xi_1} \cos \eta, \rho e^{i \xi_2} \sin \eta, 0 \ldots, 0) \in \mathbb{C}^n \ | \ \xi_1, \xi_2 \in [0,2\pi], \  \eta \in [0,\emax], \ \rho \ge 0 \}.
\end{equation}
Note that $\Cone$  contains the points $z', z''$.  Define $\tilde{L} = L \cap \Cone$.  Thus $\tilde{L}$ also contains the points $z', z''$.

\begin{lemma}
$\tilde{L}$ is a complex disk whose radius is  $O(r)$,  and $\tilde{L} \subset A_{r, 4r}$.
\end{lemma}
\begin{proof}
We first show that $\tilde{L} \subset A_{r, 4r}$.
Let 
\begin{equation}
z=(z^1, z^2, 0, \ldots, 0)= (\rho e^{i \xi_1} \cos \eta, \rho e^{i \xi_2} \sin \eta, 0 \ldots, 0) \in \tilde{L}.
\end{equation} Since $|z^1|^2+ |z^2|^2 = \rho^2$, it suffices to show that $r \le \rho \le 4r$.  We have $z^1 = \rho e^{i \xi_1} \cos \eta = 2r \cos \theta$, and taking absolute values gives
$\rho = 2r {\cos \theta}/{\cos \eta}.$
Since $\theta, \eta \in [0, \pi/6]$, we have
\begin{equation}
r \le \sqrt{3}\, r \le \rho \le \frac{4}{\sqrt{3}}r \le 4r
\end{equation}
as required.

Next, observe that $(z^1, z^2, 0 \ldots, 0) \in \tilde{L}$ if and only if 
\begin{equation}
(z^1, z^2, 0 \ldots, 0) = (2r \cos \theta, 2r \cos \theta \, e^{i \xi_2} \tan \eta, 0 \ldots, 0)
\end{equation}
 for some $\eta \in [0, \emax]$ and $\xi_2 \in [0, 2\pi]$.  Hence $\tilde{L}$ is a disk of radius 
 \begin{equation}
 r_0=2r \cos \theta \, \tan (\emax)
 \end{equation} which is of order $r$. \qed
\end{proof}

Observe also that if $j \in \mathbb{Z}_k$ then $j \cdot \tilde{L}$ for $j=0,1,2, \ldots, k-1$ are disjoint.  For later purposes, write $\tilde{L}_j = j \cdot \tilde{L}$.  By symmetry of $\omega(t)$, anything we prove about $\tilde{L}= \tilde{L}_0$ will hold also for $\tilde{L}_j$ for $j=1, 2, \ldots, k-1$.
The key lemma we will need to complete the proof of Theorem \ref{diamA} is as follows:

\begin{lemma} \label{lemma:intL}  We have for all $t$ with $0< (T-t) \le r^{\delta}$,
\begin{equation}
\int_{\tilde{L}} \omega(t)  \le C r^{\delta}.
\end{equation}
\end{lemma}

Before we prove this, we show how it implies the result of  the theorem.

\begin{proof}[Proof of Theorem \ref{diamA} given Lemma \ref{lemma:intL}]
Note that $z', z''$ lie in the disc $\tilde{L}$ which is of radius $r_0 = 2r \cos \theta \tan(\emax)$.
 Moreover, $z',z''$ lie in the square $S \subset \tilde{L}$ with the same center as $\tilde{L}$ and side of length $2r_1$ with
$r_1= 2r \sin \theta$.  (Note that $2r_1 \le 4r \sin \pi/12 < \sqrt{2} r_0$, and since a square of side $a$ fits inside a circle of radius $b$ if $a \le \sqrt{2} b$, this shows that the square $S$ does indeed lie in the disk $\tilde{L}$).

The square $S$ lies inside a plane $P \cong \mathbb{R}^2$.  We may assume that the center of $S$ corresponds to the origin in $\mathbb{R}^2$.  Let's change coordinates now, and use $x$ and $y$ for the coordinates in $\mathbb{R}^2$, so that the center of the square $S$ corresponds to $x=y=0$.
In these coordinates, we may assume that the points $z', z''$ are given by $(-r_1, 0)$ and $(r_1, 0)$ respectively.  From Lemma \ref{lemma:intL} we see that for $t$ with $(T-t) \le r^{\delta}$,
\begin{equation}
\int_S \omega(t) \le C r^{\delta}.
\end{equation}
Hence
\begin{equation}
\int_{-r_1}^{r_1} \int_{-r_1}^{r_1} (\tr{\geu}{g})(x,y)dxdy \le C r^{\delta}.
\end{equation}
Take $\ve = r \, r_1^{\delta/3}$.  Then since
\begin{equation}
\int_{-\ve}^{\ve} \int_{-r_1}^{r_1} (\tr{\geu}{g})(x,y)dxdy \le C r^{\delta},
\end{equation}
there exists $y'\in (-\ve, \ve)$ with
\begin{equation}
\ \int_{-r_1}^{r_1} (\tr{\geu}{g}) (x, y')dx \le C \frac{r^{\delta}}{\ve},
\end{equation}
Let $w', w''$ be the points $(-r_1, y')$ and $(r_1, y')$ respectively.  Let $\gamma_1$ be the straight line path in $S$ between $w'$ and $w''$.  We have
\begin{align} \nonumber
L_{g(t)} (\gamma_1) & = \int_{-r_1}^{r_1} \sqrt{g(\partial_x, \partial_x)} (x,y') dx \\ \nonumber
& \le \int_{-r_1}^{r_1} (\sqrt{\tr{\geu}{g}})(x, y') \sqrt{\geu(\partial_x, \partial_x)} (x, y') dx \\  \nonumber
& \le \left( \int_{-r_1}^{r_1} (\tr{\geu}{g})(x, y') dx \right)^{1/2} \left( \int_{-r_1}^{r_1}  \geu(\partial_x, \partial_x) (x, y') dx \right)^{1/2} \\
& \le  \frac{Cr^{\delta/2}r_1^{1/2}}{\ve^{1/2}} = C r^{\delta/2-1/2} r_1^{1/2 - \delta/6} \le C r^{\delta/3},
\end{align}
using the fact that $r_1 \le r$.
On the other hand, since $g \le r^{-2} \geu$, we may estimate the length of the straight line path $\gamma_2$ in $S$ between $z'$ and $w'$ by
\begin{equation}
L_{g(t)} (\gamma_2) = \int_{-\ve}^{\ve} \sqrt{g(\partial_y, \partial_y)}(-r_1, y) dy \le \frac{C\ve}{r} \le C r^{\delta/3}.
\end{equation}
Similarly, if we let $\gamma_3$ be the straight line path in $S$ between $w''$ and $z''$ we have
\begin{equation}
L_{g(t)} (\gamma_3) \le C r^{\delta/3}.
\end{equation}
Thus, if $\gamma$ is the piecewise smooth path in $S$ from $z'$ to $z''$ obtained by going along $\gamma_2$ then $\gamma_1$ and finally along $\gamma_3$, we see that
\begin{equation}
L_{g(t)} (\gamma) = L_{g(t)} (\gamma_1) + L_{g(t)} (\gamma_2) + L_{g(t)} (\gamma_3) \le C r^{\delta/3}.
\end{equation}
Since the path $\gamma$ lies in $S \subset \tilde{L} \subset A_{r,4r}$ we have 
established the estimate (\ref{eqn:diamA1}).

For (\ref{eqn:diamA2}), observe that, since $\omega(t)$ converges to $\omega_T$ uniformly in $C^{\infty}$ on compact subsets of $B \setminus \{0 \}$, we may let $t \rightarrow T$ in Lemma \ref{lemma:intL} to obtain
\begin{equation}
\int_{\tilde{L}} \omega_T \le C r^{\delta}.
\end{equation} 
Then we repeat word for word the argument given above with $g(t)$ replaced by $g_T$, and we obtain a path $\gamma$ between $z'$ and $z''$ with $L_{g_T} (\gamma) \le C r^{\delta}$.  This completes the proof of Theorem \ref{diamA}. \qed
\end{proof}

It now remains to prove Lemma \ref{lemma:intL}.  We will do this by exhibiting $\cup_{j=0}^k \tilde{L}_j$ as a piece of the  boundary of a certain $3$-chain inside $\tilde{B}$ and then applying Stokes' theorem.

First define a subset $F$ of $\mathbb{C}^n$ by
\begin{equation}
F = \{ ( 2ru \cos \theta, 2ru e^{iv} \cos \theta \tan (\emax), 0, \ldots, 0) \in \mathbb{C}^n \ | \ u \in [0,1], v \in [0,2\pi] \}.
\end{equation} 
Note that the points in $F$ with $u=1$ coincide with the boundary of the disk $\tilde{L}$.  Thus $F$ connects the boundary of $\tilde{L}$ to the origin in $\mathbb{C}^n$.  If we exclude $u=0$ and $u=1$ in $F$ then $F$ is a real submanifold of $\mathbb{C}^n$ of dimension 2.  We have:

\begin{lemma} \label{lemma:intF}
There exists a uniform constant $C$ such that
\begin{equation}
 \left| \int_F \omega(t)  \right| \le C r^{\delta}.
\end{equation}
\end{lemma}
\begin{proof}
To prove this, we first calculate the volume of $F$ with respect to a certain Riemannian metric which we now define.  Let $\tilde{g}$ be the Riemannian metric on $\mathbb{C}^n= \mathbb{R}^{2n}$ given by
\begin{equation}
\tilde{g} = \frac{1}{\rho^{2(1-\delta)}}d \rho^2 + g_{S^{2n-1}},
\end{equation}
where $\rho$ is a radial coordinate and $g_{S^{2n-1}}$ is the standard metric on $S^{2n-1}$.  (Recall that the Euclidean metric is $d\rho^2 + \rho^2 g_{S^{2n-1}}$).  

A straightforward computation shows that if $dV_{\tilde{g}|F}$ is the volume form on $F$ induced by $\tilde{g}$, then
\begin{equation} \label{intF}
\int_F dV_{\tilde{g}|F} = \frac{2 \pi}{\delta} \sin (\emax) \left( \frac{2r \cos \theta}{\cos (\emax)} \right)^{\delta},
\end{equation}
and hence is of order $O(r^{\delta})$.

Next we make the observation that if $\omega$ is a K\"ahler form on $\mathbb{C}^n$ with associated metric $g$ and $Y$ a real submanifold of dimension 2 then
\begin{equation} \label{vf}
- dV_{g|Y} \le \omega|_Y \le  dV_{g|Y},
\end{equation}
where $V_{g|Y}$ is the volume form induced by $g$ on $Y$.
It is a well-known fact for submanifolds in a K\"ahler manifold. To see this, it suffices to compute at a single point $p$.  We may choose an orthonormal basis $e_1, \ldots, e_{2n}$ for the tangent space of $\mathbb{C}^n$ at $p$ with respect to $g$ such that
\begin{equation}
\omega|_p = e_1^* \wedge e_2^* + \cdots  + e_{2n-1}^* \wedge e_{2n}^*.
\end{equation}
Now choose two unit orthogonal vectors $e$, $e'$ (with respect to $g$) lying in $T_pY$, which we may write as
\begin{equation}
e= \sum_{i=1}^{2n} a_i e_i, \ e' = \sum_{i=1}^{2n} a'_i e_i,
\end{equation}
Then $|dV_{g|Y} (e, e')| =1$.  And
\begin{align} \nonumber | \omega (e, e') | & = | a_1 a_2' - a_2 a_1' + \cdots +  a_{2n-1} a_{2n}' - a_{2n} a'_{2n-1}| \\ \nonumber
&\le |a_1 a'_2|+|a_2 a'_1|+ \cdots+ |a_{2n-1} a'_{2n}|+|a_{2n} a'_{2n-1}| \\ 
&\le (|a_1|^2 + \cdots +|a_{2n}|^2) ^{1/2} ( |a_1'|^2 + \cdots + |a'|_{2n}^2)^{1/2} =1 
\end{align}
since $e,e'$ are unit vectors. The claim (\ref{vf}) follows.

We can now finish the proof of the lemma.  From Lemma \ref{lemmaest2} we have $g(t) \le C \tilde{g}$ in $B$, for some uniform constant $C$.  Hence $dV_{g(t)|F} \le C' dV_{\tilde{g}|F}$ and
applying (\ref{vf}), (\ref{intF}),
\begin{equation}
\left| \int_F \omega(t) \right| \le    \int_F dV_{g(t)|F} \le C' \int_F dV_{\tilde{g}|F} \le C'' r^{\delta},
\end{equation}
as required. \qed
\end{proof}

As with $\tilde{L}$ it will be convenient to write $F_j = j \cdot F$, where $j=0, 1, 2, \ldots, k-1 \in \mathbb{Z}_k$.  By the invariance of $\omega(t)$, Lemma \ref{lemma:intF} holds with $F$ replaced by $F_j$ for $j=1, 2, \ldots, k-1$.  Finally:

\begin{proof}[Proof of Lemma \ref{lemma:intL}]
To finish the proof of Lemma \ref{lemma:intL}, we need to work on the blow-up $\tilde{B}$.   
 In fact, we are only interested in a complex submanifold $\tilde{B}_2$ of $\tilde{B}$ of complex dimension 2,
 defined by 
\begin{eqnarray} \nonumber
\tilde{B}_2 &  = & \{ ([Z_1, Z_2, 0, \ldots, 0], (z^1, z^2, 0, \ldots, 0)) \in \mathbb{P}^{n-1} \times B \ | \\ && \ \ \sigma \textrm{ is in the line } \lambda \mapsto \lambda( Z_1^k, Z_2^k, 0, \ldots, 0) \}
\end{eqnarray}
 Write $\pi_2$ for the restriction of $\pi$ to $\tilde{B}_2$ and $N:=\pi_2^{-1} (0)$, a projective line in the exceptional divisor $E$.

In the chart $\{ Z_1 \neq 0 \}$, we use local coordinates $y$ and $w$ given by
\begin{equation}
\sigma = \frac{y}{Z_1^k} (Z_1^k, Z_2^k, 0, \ldots, 0) \quad \textrm{and} \quad w = \frac{Z_2}{Z_1}.
\end{equation}
Note that $\sigma = (y, yw^k, 0, \ldots, 0)$.
 Recall that the disk $\tilde{L}$ is given by points 
 \begin{equation}
 (z^1, z^2, 0, \ldots,0)=(2r\cos \theta, 2r e^{iv} \cos \theta \tan \eta, 0 \ldots, 0) \in B_2
 \end{equation}
  with $\eta \in [0, \emax]$ and $v \in [0,2\pi]$.  We consider the line segments
\begin{equation}
u \in [0,1] \mapsto (2ru \cos \theta, 2r u e^{iv} \cos \theta \tan \eta, 0, \ldots, 0) \in B_2,
\end{equation}
and also the corresponding line segments we obtain by applying $j \in \mathbb{Z}_k$ for $j=1, \ldots, k-1$.
In the blow-up $\tilde{B}_2$, using the coordinates $y,w$, these line segments correspond to
\begin{equation}
u \in [0,1] \mapsto (y, w) = (2ru \cos \theta,  e^{\frac{iv}{k}}  \left( \tan \eta\right)^{1/k}) \in \tilde{B}_2.
\end{equation}
The line segments intersect $N \subset \tilde{B}_2$ at the points
\begin{equation}
W = \{ (0, [1, e^{\frac{iv}{k}} \left( \tan \eta \right)^{1/k}, 0 \ldots, 0 ] ) \ | \ v \in [0, 2\pi], \ \eta \in [0, \emax] \} \in \tilde{B}_2.
\end{equation}
The interior of $W$ is a submanifold of $\tilde{B}_2$ of real dimension two, contained within the complex submanifold $N$.  We have from (\ref{intN}),
\begin{equation}
0 \le \int_W \omega(t) \le \int_{N} \omega(t) \le C(T-t).
\end{equation}
Thus we have
\begin{equation} \label{W}
\left|  \int_W \omega(t) \right|  \le C r^{\delta} \quad \textrm{for all } t \ \textrm{with } 0<T-t \le r^{\delta}.
\end{equation}

Now $W \cup \bigcup_{j=0}^{k-1} F_j \cup  \tilde{L}_j$ defines a boundary in $\tilde{B}$.  Indeed, it is the boundary of the set 
\begin{eqnarray} \nonumber
M & = & \mathbb{Z}_k \cdot \{ ( 2ru \cos \theta, 2ru e^{iv} \cos \theta \tan \eta, 0, \ldots, 0) \in \mathbb{C}^n \ | \\ && \mbox{} \ u \in [0,1], v \in [0,2\pi], \eta \in [0, \emax] \}.
\end{eqnarray}
 Hence by Stokes' Theorem, since $\omega(t)$ is closed,
\begin{equation}
\int_{W \cup \bigcup_{j=0}^{k-1} F_j \cup  \tilde{L}_j} \omega(t)=0.
\end{equation}
Thus combining Lemma \ref{lemma:intF} with (\ref{W}) we have
\begin{equation}
\int_{\bigcup_{j=0}^{k-1} \tilde{L}_j} \omega(t)  \le Cr^{\delta},
\end{equation}
and  Lemma \ref{lemma:intL} follows. \qed
\end{proof}

\section{The Metric Completion}

In this section, we give the proof of parts (ii) and (iii) of Theorem \ref{mainthm}.  Again, we assume that there is just one $(-k)$ exceptional divisor $E$ which is mapped by $\pi$ to $y_0$.  
 
As in \cite{SW2}, we define a distance function $d_T$ on $Y$ as follows.  Since $\pi$ is an isomorphism from $X\setminus E$ to $Y \setminus \{y_0 \}$, we may regard $g_T$ as a smooth K\"ahler metric on $Y \setminus \{ y_0\}$.  We extend $g_T$ to a (possibly discontinuous) metric $\tilde{g}_T$ on $Y$ by setting $\tilde{g}_T$ to be zero at $y_0$.  Then we define a distance function $d_T: Y \times Y \rightarrow \mathbb{R}$ by
\begin{equation}
d_T(y_1, y_2) = \inf_{\gamma}  \int_{0}^1  \sqrt{ \tilde{g}_T ( \gamma'(s), \gamma'(s))}  ds,
\end{equation}
where the infimum is taken over all piecewise smooth paths $\gamma: [0,1] \rightarrow Y$ with $\gamma(0)=y_1$, $\gamma(1)=y_2$.

\begin{proposition} $(Y, d_T)$ is a compact metric space and 
$(X, g(t))$ converges to $(Y, d_T)$ in the Gromov-Hausdorff sense as $t \rightarrow T^+$.
\end{proposition}
\begin{proof}
The proof is the same as in \cite{SW2}. \qed
\end{proof}

The main result of this section is:

\begin{proposition}
$(Y, d_T)$ is the metric completion of $(Y \setminus \{ y_0 \}, d_{g_T})$, where we are writing $d_{g_T}$ for the distance function on $Y \setminus \{ y_0 \}$ induced by $g_T$.
\end{proposition}
\begin{proof}  First observe that the inequality 
\begin{equation} \label{other} 
d_{g_T} \ge d_T|_{Y \setminus \{y_0\}}.
\end{equation}
follows immediately from the definition.  

We will show:
\begin{equation} \label{firstineq}
d_{g_T} \le d_T|_{Y \setminus \{y_0\}}.
\end{equation}
To see this, let $p, q \in Y \setminus \{ y_0 \}$, and let $\ve>0$.  By the definition of $d_T$, there exists a path $\gamma:[0,1] \rightarrow Y$ from $p$ to $q$ such that 
\begin{equation}
\int_0^1 \sqrt{ \tilde{g}_T (\gamma'(s), \gamma'(s))} ds \le d_T(p,q) + \ve/4.
\end{equation}
If $\gamma$ does not pass through the origin then we have
\begin{eqnarray} \nonumber
d_{g_T}(p,q) & \le & \int_0^1 \sqrt{ g_T (\gamma'(s), \gamma'(s))} ds \\  \nonumber
& = & \int_0^1 \sqrt{ \tilde{g}_T (\gamma'(s), \gamma'(s))} ds \\ \label{1dT}
& \le &  d_T(p,q) + \ve/4 < d_T(p,q) + \ve.
\end{eqnarray}
Assume instead that $\gamma$ \emph{does} pass through the origin.
We identify a neighborhood of $y_0$ in $Y$ with a small ball $B$ in $\mathbb{C}^n$.
Applying Corollary \ref{cor:diamgT}, we see that there exists a small constant $c>0$ such that $\diam_{g_T} ( B_c \setminus \{0 \}) < \ve/4$, where $B_c$ denotes the ball of radius $c$ with respect to the Euclidean metric on $B$.   Shrinking $c$ if necessary, assume that $p,q \notin \overline{B_c}$.  Let $\gamma (s_1)$ be the point where $\gamma$ first enters $\overline{B_{c}}$ (we know it enters, since $\gamma$ passes through the origin).  Let $\gamma(s_1)$ be the point at which $\gamma$ last leaves $\overline{B_c}$.  Then since $\diam_{g_T} ( \overline{B_c} \setminus \{0 \}) \le \ve/4$, we can find a path $\sigma$ from $\gamma(s_1)$ to $\gamma(s_2)$ contained in $\ov{B_c}$ with the length of $\sigma$ less than $\ve/2$.  That is:
 \begin{equation}
 \int_0^1 \sqrt{ g_T (\sigma'(s), \sigma'(s))} ds \le \ve/2.
 \end{equation}
The path in $Y \setminus \{y_0 \}$ constructed by going first along $\gamma$ from $p$ to $\gamma(s_1)$; second along $\sigma$ from $\gamma(s_1)$ to $\gamma(s_2)$; and  third along $\gamma$ from $\gamma(s_2)$ to $q$, has length $L$, say, with respect to $g_T$ with   $L \le d_T(p,q) + \ve/4 + \ve/2.$  Thus we have 
\begin{equation}
d_{g_T}(p,q) \le L < d_T (p,q) + \ve.
\end{equation}
Combining with (\ref{1dT}), we see that since $p,q$ were any points in $Y\setminus \{ y_0\}$ and $\ve>0$ was arbitrary, this proves (\ref{firstineq}).

Thus $d_{g_T} =  d_T|_{Y \setminus \{y_0\}}$.  Since $(Y, d_T)$ is complete, it follows that it is the metric completion of $( Y \setminus \{0 \}, d_{g_T})$. \qed
\end{proof}

Thus we have proved parts (ii) and (iii) of Theorem \ref{mainthm}.

\pagebreak[3]
\section{Continuing the flow on $Y$} \label{sectiv}

In this section we first show the existence of a solution of the K\"ahler-Ricci flow on the orbifold $Y$, which converges in $C^{\infty}$ on compact subsets outside the orbifold points to the metric $g_T$.  This will establish part (iv) of Theorem \ref{mainthm}.  We give only a brief outline of the proof here, since the arguments in \cite{SW2} are very similar.

As above, we will assume there is only one exceptional divisor $E$, blowing down to the point $y_0 \in Y$.  As in Corollary 4.1 of  \cite{SW2}, we have $C^{\infty}$ estimates  for $g(t)$ on $X \setminus E$ for $t \in [0,T)$.  The estimate for the $k$th derivative degenerates like $|s|^{-2\alpha_k}_h$ for some $\alpha_k>0$.


From Lemma \ref{basicestimates}, there is a closed positive $(1,1)$-current $\omega_T$ and a bounded function $\varphi_T$ with
\begin{equation}
\omega_T = \pi^* \omega_Y + \ddbar \varphi_T.
\end{equation}
Moreover, $\varphi_T$ is constant on $E$ and hence we can write $\varphi_T = \pi^* \psi_T$ for some bounded function $\psi_T$ on $Y$ which is smooth on $Y \setminus \{ y_0 \}$.  Define
\begin{equation}
\omega' = \omega_Y + \ddbar \psi_T.
\end{equation}
Then $\omega'$ is a closed positive $(1,1)$-current on $Y$ which is smooth on $Y \setminus \{ y_0 \}$ so that in particular $\psi_T$ extends to a plurisubharmonic function on $Y$ with respect to $\omega_Y$.

Since $\omega_T^n/(\pi^* \omega_Y)^n$ is in $L^p(X)$ for some $p>1$ with respect to the measure $(\pi^* \omega_Y)^n$ (see Lemma 5.2 of \cite{SW2}) if follows from 
  it follows from \cite{Kol1, Zha1, EGZ} that $\varphi_T$ and hence $\psi_T$ is continuous.

We now construct a family of reference $(1,1)$-forms on $Y$.  Let
 $\Oorb$ be an orbifold volume form on $Y$ which coincides on the ball $B$ with the Euclidean volume form.  Define
 \begin{equation}
\chi := \ddbar \log \Oorb \in c_1(K_Y),
\end{equation}
so that, in particular, $\chi=0$ on $B$ and $\pi^* \chi$ is a smooth closed (1,1) form on $X$.  

There exists $T'>T$ such that for $t\in[T, T']$, the closed (1,1) form
\begin{equation} \label{hatotY}
\hat{\omega}_{Y,t} =  \omega_Y + (t-T) \chi
\end{equation}
is nonnegative on $Y$ and positive definite on $Y \setminus \{ y_0 \}$ and  $[\hat{\omega}_{Y,t}]$ is a K\"ahler class on $Y$.  Then we construct a family of functions $\psi_{T, \ve}$ on $Y$ which converge to $\psi_T$, following the method of \cite{SoT3}.  For $\ve>0$ sufficiently small and $K$ fixed and sufficiently large, we define a family of volume forms $\Omega_{\ve}$ on $Y$ by
\begin{equation}
\Omega_{\ve} = \frac{|s|_h^{2K} \omega^n(T-\ve)}{\ve + |s|_h^{2K}} + \ve \Oorb, \quad \textrm{on } Y\setminus \{y_0 \},
\end{equation}
and $\Omega_{\ve}|_{y_0} = \ve \Oorb|_{y_0}$.  Define $\psi_{T, \ve}$ to be the unique solution of
\begin{equation} \label{mag}
(\omega_Y + \ddbar \psi_{T, \ve})^n = C_{\ve} \Omega_{\ve}, \quad \sup_Y (\psi_{T, \ve} - \psi_T) = \sup_Y (\psi_{T} - \psi_{T,\ve}),
\end{equation}
with constants $C_{\ve}$ such that $C_{\ve} \int_Y \Omega_{\ve} = \int_Y \omega_Y^n$.  The existence of solutions to (\ref{mag}) follows from the results of \cite{Kob, Zha1, EGZ}.  Moreover, $\psi_{T, \ve}$ lies in $C^k(Y) \cap C^{\infty}(Y\setminus \{ y_0 \})$ for some $k$ which can be chosen to be sufficiently large by raising the value of $K$.

By the definition of $\Omega_{\ve}$ and a generalization of Kolodziej's stability theorem \cite{Kol2, EGZ} we have $\psi_{T, \ve} \rightarrow \psi_T$ in $L^{\infty}(Y)$ as $\ve \rightarrow 0$.  We let $\varphi_{\ve}$ solve
\begin{equation} \label{eqphie}
\frac{\partial \varphi_{\ve}}{\partial t} = \log \frac{(\hat{\omega}_{Y,t} + \ddbar \varphi_{\ve})^n}{\Oorb}, \quad \varphi|_{t=T} = \psi_{T, \ve},
\end{equation}
for $t$ in $[T, T']$.  It follows (with only minor changes from the arguments of  \cite{SW2}) that as $\ve \rightarrow 0$, $\varphi_{\ve}$ converge in $L^{\infty}([T,T'] \times Y)$ to a function $\varphi \in C^0([T, T'] \times Y) \cap C^{\infty}(([T, T'] \times Y) \setminus \{ (T, y_0) \})$.  Moreover, the convergence is smooth on compact subsets of $([T, T'] \times Y) \setminus \{ (T, y_0) \}$ and we have: (cf. \cite{SoT3})

\begin{lemma} \label{lphi}
There exists a unique $\varphi \in C^0([T, T'] \times Y) \cap C^{\infty}(([T, T'] \times Y) \setminus \{ (T, y_0) \})$ satisfying $\varphi|_{t=T} = \psi_T$ and
\begin{equation}
\hat{\omega}_{Y,t} + \ddbar \varphi>0, \quad \frac{\partial \varphi}{\partial t} = \log \frac{ (\hat{\omega}_{Y,t} + \ddbar \varphi)^n}{\Omega_{\emph{orb}}} \quad \textrm{on } \ ([T, T'] \times Y) \setminus \{ (T, y_0) \}.
\end{equation}
\end{lemma}

By definition of $\hat{\omega}_{Y,t}$, the solution $\omega(t)=\hat{\omega}_{Y,t}+\ddbar \varphi$ then solves the K\"ahler-Ricci flow on $Y$ in the orbifold sense for $t \in (T, T')$ and $\omega(t)$ converges to $\omega_T$ as $t \rightarrow T^+$ in $C^{\infty}$ on compact subsets of $Y \setminus \{ y_0 \}$.  We may extend $g(t)$ beyond $T'$ to be a maximal solution  of the flow.  This establishes (iv) of Theorem \ref{mainthm}.

We end this section by establishing a more general uniqueness lemma, which we will use later.

\begin{lemma} \label{lemmau} If  $\tilde{\varphi} \in L^\infty([T, T']\times Y) \cap C^{\infty}( [T, T']\times (Y\setminus \{y_0 \})) $ satisfies 
\begin{equation}
\tilde{\varphi}|_{t=T}=\psi_T \qquad \ddt{\tilde{\varphi}} = \log \frac{(\hat{\omega}_{Y,t} + \ddbar \tilde{\varphi})^n}{\Omega_{\emph{orb}}},  \qquad \textrm{on } [T, T']\times (Y \setminus\{y_0 \}),
\end{equation} then
$\tilde{\varphi}$ is equal to the function $\varphi$ of Lemma \ref{lphi} on $[T, T'] \times Y \setminus \{y_0\}$.
\end{lemma}

The key point in this lemma is that \emph{a priori} $\tilde{\varphi}$ may be singular at $y_0$ for values of $t$ in $(T, T']$.  We follow a method similar to that in \cite{SoT3}.
Fix $\delta>0$ sufficiently small.  For $\ve \in (0,\delta)$ and $s \in (-\delta, \delta)$ we consider $\varphi_{s,\ve}$ a solution of the following parabolic Monge-Amp\`ere equation:
\begin{equation}
\ddt{\varphi_{s, \ve}} = \log \frac{ (\hat{\omega}_{Y,t}+s\omega_Y+ \ddbar \varphi_{s, \ve})^n }{\Oorb}, ~~~\varphi_{s,\ve} |_{t=T}=(1+s) \psi_{T, \ve},
\end{equation}
for $t \in [T,T'']$, some $T'' \in (T, T']$.   Write $\omega_{s, \ve}= \hat{\omega}_{Y,t}+s\omega_Y+ \ddbar \varphi_{s, \ve}$.
 The solution $\varphi_{s, \ve}$ is smooth for $\ve>0$ and we have uniform $C^{\infty}$ estimates (independent of $\ve$ and $s$) away from the point $(T,y_0)$.

\begin{lemma} There exists $C>0$ such that on $[T, T'']\times Y$,
\begin{enumerate}
\item[(i)] $|\varphi_{s, \ve}|\leq C.$
\item[(ii)] $\displaystyle{\left| \frac{\partial \varphi_{s, \ve}}{\partial t} \right| \leq C}.$
\item[(iii)] $\displaystyle{\left| \frac{\partial \varphi_{s, \ve}}{\partial s} \right| \leq C}.$
\end{enumerate}
\end{lemma}
\begin{proof} Part (i) follows easily from the maximum principle.  
For (ii), compute
\begin{equation}
\left( \ddt{} - \Delta_{\omega_{s, \ve}} \right) \frac{\partial \varphi_{s, \ve}}{\partial t} = \tr{\omega_{s,\ve}}{\chi} \ge - C \tr{\omega_{s, \ve}}{\omega_Y}.
\end{equation}
Then if $Q = \partial \varphi_{s, \ve}/\partial t + A \varphi_{s, \ve}$ with $A$ sufficiently large, we have
\begin{eqnarray} \nonumber
\left( \ddt{} - \Delta_{\omega_{s, \ve}} \right) Q & = & \tr{\omega_{s, \ve}}{\chi} + A \ddt{\varphi_{s,\ve}} - A \tr{\omega_{s,\ve}}{(\omega_{s,\ve} - \hat{\omega}_{Y,t} - s \omega_Y)} \\ 
& \ge & - C + AQ.
\end{eqnarray} 
By a maximum principle argument, $Q$ is bounded from below, giving a lower bound for $\partial \varphi_{s,\ve}/\partial t$.  The upper bound can be proved similarly.

For (iii), compute
\begin{equation}
\left( \ddt{} - \Delta_{\omega_{s, \ve}} \right) \frac{\partial \varphi_{s, \ve}}{\partial s} = \tr{\omega_{s,\ve}}{\omega_Y} \geq 0.
\end{equation}
The lower bound of $\partial \varphi_{s,\ve}/ \partial s$ follows immediately.  For the upper bound, it suffices to obtain an upper bound of the quantity
 \begin{equation}
 H= \frac{\partial \varphi_{s, \ve}}{\partial s} -A\varphi_{s,\ve},
 \end{equation}
for some constant $A$.  Choosing $A$ sufficiently large we have, using (ii),
\begin{eqnarray}
\left( \ddt{} - \Delta_{\omega_{s, \ve}} \right) H & = & \tr{\omega_{s, \ve}}{\omega_Y} - A \ddt{\varphi_{s,\ve}} + A \tr{\omega_{s,\ve}}{(\omega_{s,\ve} - \hat{\omega}_{Y,t} - s \omega_Y)} \le C, 
\end{eqnarray}
and the upper bound of $H$  follows. \qed
\end{proof}

It follows that $\varphi_{s, \ve}$ converges in $L^\infty(Y)$ to a function $\varphi_{0, \ve}$ as $s \rightarrow 0$.  Moreover, the convergence is smooth away from $(T, y_0)$ and so $\varphi_{0,\ve}$ solves (\ref{eqphie}) and hence must equal $\varphi_{\ve}$.

\begin{proof}[Proof of Lemma \ref{lemmau}]  It suffices to prove the result for the shorter time interval $[T,T'']$.  Let $\sigma$ be a defining section of the exceptional divisor $E$ and $h$ be a smooth Hermitian metric equipped on the associated line bundle such that $\omega_Y- \delta R(h)>0$ for sufficiently small $\delta>0$. We also assume that $|\sigma|^2_h\leq 1$.

First, restrict to $s>0$ and define a smooth function $u_{s, \ve}$ on $[T, T''] \times Y \setminus \{ y_0 \}$ by
\begin{equation}
u_{s, \ve} = \varphi_{s, \ve} - \tilde{\varphi} - s^2 \log |\sigma|^2_h.
\end{equation}  Observe that $u_{s, \ve}(y) \rightarrow \infty$ as $y \rightarrow y_0$.   Compute
\begin{eqnarray} \nonumber
\ddt{u_{s,\ve}} & = & \log \frac{(\hat{\omega}_{Y,t} + s \omega_Y + \ddbar \varphi_{s, \ve})^n}{(\hat{\omega}_{Y,t} + \ddbar \tilde{\varphi})^n} \\ \label{rhs}
& = & 
 \log \frac{( \hat{\omega}_{Y,t} + \ddbar \tilde{\varphi} + \ddbar u_{s,\ve} + s(\omega_Y - sR(h) )   )^n}{(\hat{\omega}_{Y,t} + \ddbar\tilde{\varphi})^n}.
\end{eqnarray}
If $u_{s, \ve}$ attains a minimum on $Y \setminus \{ y_0 \}$ then at that point the right hand side of (\ref{rhs}) is nonnegative.
Hence by the maximum principle,
\begin{equation}
u_{s,\ve} \geq \inf_Y u_{s, \ve}|_{t=T}   \ge \inf_Y \{(\psi_{T,\ve}-\psi_T)+ s\psi_{T,\ve}\}
\end{equation} and so
\begin{equation}
\varphi_{s,\ve} \geq \tilde{\varphi} + \inf_Y \{(\psi_{T,\ve}-\psi_T)+ s\psi_{T,\ve}\}+ s^2 \log |\sigma|^2_h.
\end{equation}
Letting $s, \ve \rightarrow 0$, we have $\varphi \geq \tilde{\varphi}.$

For the other inequality consider $s<0$ and define 
\begin{equation}
v_{s, \ve} = \varphi_{s, \ve} - \tilde{\varphi} + s^2 \log |\sigma|^2_h,
\end{equation}
  which has $v_{s, \ve} (y) \rightarrow -\infty$ as $y \rightarrow y_0$.  Then
\begin{equation}
\ddt{v_{s,\ve}} = \log \frac{( \hat{\omega}_{Y,t} + \ddbar \tilde{\varphi} + \ddbar v_{s,\ve} - (-s)(\omega_Y - (-s) R(h) )   )^n}{(\hat{\omega}_{Y,t} + \ddbar\tilde{\varphi})^n}.
\end{equation}
Arguing again using the maximum principle, we have
\begin{equation}
v_{s,\ve} \leq\sup_Y v_{s, \ve}|_{t=T}= \sup_Y \{(\psi_{T,\ve}-\psi_T)+ s\psi_{T,\ve}\} \end{equation} 
and so
\begin{equation} 
\varphi_{s,\ve} \leq \tilde{\varphi} + \sup_Y \{(\psi_{T,\ve}-\psi_T)+ s\varphi_{T,\ve}\}  - s^2 \log |\sigma|_h^2.
\end{equation}
By letting $s, \ve \rightarrow 0$, we have $\varphi\leq \tilde{\varphi},$
and this completes the proof of the lemma.  \qed
\end{proof}

\pagebreak[3]
\section{Estimates near the orbifold point after time $T$} \label{sectafter}

In this section we prove (v) of Theorem \ref{mainthm} by constructing a solution of the K\"ahler-Ricci flow on $Y$ by a different method.  
We do this using a family of flows on the original manifold $X$.   For simplicity we again assume that $\pi: X \rightarrow Y$ contracts only a single  exceptional divisor $E$ on $X$ to a point $y_0$.   For the sake of readability we include in this section some arguments which are similar to those given in Section 6 of \cite{SW2}.

As in Section \ref{sectiv}, we consider the reference forms $\hat{\omega}_{Y,t}= \omega_Y + (t-T) \chi$ for $t \in [T, T']$ given by (\ref{hatotY}), where $\chi = \ddbar \log \Oorb$ vanishes on $B$.  
  Define a smooth function $\gamma$ on $Y$ such that 
\begin{equation}
\gamma(z) = r^{2(n-k)} \ \ \textrm{for } z \in B,
\end{equation}
for $r^2 = |z^1|^2 + \cdots  + |z^n|^2$, with $\gamma$ positive outside $B$.
Then, using the definition of $\pi$, one can see that there exists a smooth volume form $\Omega_X$ on $X$ with
\begin{equation}
\pi^* \Oorb = (\pi^* \gamma) \Omega_X.
\end{equation}
Note that $\pi^* \gamma$ is not smooth along $E$ in general.

In this section we will prove the existence of an \emph{a priori} different and possibly non-smooth orbifold solution $\tilde{\varphi}$ of the parabolic complex Monge-Amp\`ere equation
\begin{equation} \label{maY}
\ddt{\tilde{\varphi}} = \log \frac{ ( \hat{\omega}_{Y,t} + \ddbar \tilde{\varphi} )^n }{ \Oorb}, ~~~\tilde{\varphi} |_{t=T}= \psi_T,
\end{equation}
on $[T,T'] \times (Y \setminus \{ y_0 \})$, by considering a family of Monge-Amp\`ere flows on $X$.  We want $\tilde{\varphi}= \tilde{\varphi}(t,y)$ to be in  $L^{\infty}([T,T'] \times Y)$ and smooth on $[T,T'] \times (Y \setminus \{ y_0 \})$.  Note that, at this stage, we do not require the solution to be smooth in a neighborhood of $y_0$ for $t > T$.


Let $\gamma_{\ve}$ be a family of smooth functions on $Y$ converging smoothly to $\gamma$ and satisfying
\begin{equation}
\gamma_\ve(z) = (\ve + r^{2k})^{(n-k)/k} \ \ \textrm{for } z \in B.
\end{equation}
By the discussion in Section \ref{sectlocalmodel}, $\pi^*\gamma_\ve$ is smooth on $X$ for $\ve >0$.



Note that if $\ve>0$ is sufficiently small then $\hat{\omega}_{Y,t} - \frac{\ve}{T} \omega_Y$ is nonnegative on  $Y$ and positive definite on $Y \setminus \{y_0\}$, for $t \in [T, T']$.  We consider for $\ve \in (0, \ve_0)$, with $\ve_0>0$ sufficiently small,
 the following family of  Monge-Amp\`ere flows on $X \times [T, T']$:
\begin{equation} \label{maX1}
\ddt{\varphi_{\ve}} = \log \frac{(  \pi^* (\hat{\omega}_{Y,t}- \frac{\ve}{T} \omega_Y) + \frac{\ve}{T} \omega_0  + \ddbar \varphi_{\ve})^n}{(\pi^*\gamma_{\ve}) \Omega_X }, ~~~ \varphi_{\ve} |_{t=T} = \varphi(T-\ve).
\end{equation}
Observe that at $t=T$, the (1,1) form $\pi^* (\hat{\omega}_{Y,t}- \frac{\ve}{T} \omega_Y) + \frac{\ve}{T} \omega_0$ is equal to $\hat{\omega}_{T-\ve}$.
Also,  if $\ve \rightarrow 0$ then we have pointwise convergence
\begin{equation}
(\pi^*\gamma_{\ve}) \Omega_X \rightarrow \pi^*\Oorb, \quad
\pi^* (\hat{\omega}_{Y,t}- \frac{\ve}{T} \omega_Y) + \frac{\ve}{T} \omega_0 \longrightarrow \pi^*  \hat{\omega}_{Y,t}, \quad \textrm{and } \varphi(T- \ve) \rightarrow \pi^* \psi_T.
\end{equation}
Hence if we have $C^{\infty}$ estimates of $\varphi_{\ve}$ on compact subsets of  $[T,T'] \times (X \setminus E)$, then the functions $(\pi|_{X\setminus E}^{-1})^* \varphi_{\ve}$ converge to a solution $\tilde{\varphi}$ of (\ref{maY}) on $[T, T'] \times Y\setminus \{y_0\}$.

We now derive uniform estimates for solutions $\varphi_{\ve}$ of (\ref{maX1}).

\begin{lemma}There exists $C>0$ such that for all $\ve \in(0,\ve_0)$
\begin{equation}
|\varphi_{\ve}| \leq C,
\end{equation}
on $[T, T']\times X$,

\end{lemma}

\begin{proof}
Choose $\alpha, \beta>0$ such that for all $t\in [T, T']$, 
\begin{equation}
\alpha \omega_Y \leq \hat{\omega}_{Y,t} - \frac{\ve}{T} \omega_Y \leq \beta \omega_Y.
\end{equation}
For the lower bound of $\varphi_{\ve}$, let $\rho_{\ve}$ be the smooth solution of the following complex Monge-Amp\`ere equation on $X$:
\begin{equation}
(\alpha \pi^* \omega_Y + \frac{\ve}{T} \omega_0  + \ddbar \rho_\ve)^n = C_{\ve} (\pi^*\gamma_{\ve}) \Omega_X, \quad \sup_X \rho_\ve =0,
\end{equation}
where $C_{\ve}$ is a constant chosen so that the integrals of the left and right sides match.  Observe that $C_{\ve}>0$ is uniformly bounded from above and below away from zero.
Applying the result of \cite{Zha1, EGZ}   we see that $| \rho_{\ve}|$ is uniformly bounded for all $\ve \in (0, \ve_0)$.

Let $v_\ve= \varphi_{\ve} - \rho_\ve$.  Then
\begin{equation}
 \ddt{v_\ve} = \log \frac{ C_{\ve}( \alpha \pi^*\omega_Y + \frac{\ve}{T} \omega_0 + \ddbar\rho_\ve + \ddbar v_\ve + \pi^*(\hat{\omega}_{Y,t} - \frac{\ve}{T} \omega_Y-\alpha \omega_Y))^n}      {(\alpha \pi^*\omega_Y  + \frac{\ve}{T} \omega_0 + \ddbar\rho_\ve)^n}.
\end{equation}
Applying the maximum principle we see that $v_{\ve}$ is uniformly bounded from below, independent of $\ve$.  This gives the lower bound for $\varphi_{\ve}$.
For the upper bound of $\varphi_{\ve}$ we argue similarly.  
\qed
\end{proof}

Define K\"ahler metrics $\omega_{\ve}$ on $X\times [T,T']$ by
\begin{equation}
\omega_{\ve} =  \pi^* (\hat{\omega}_{Y,t}- \frac{\ve}{T} \omega_Y) + \frac{\ve}{T} \omega_0  + \ddbar \varphi_{\ve}.
\end{equation}
Observe that 
\begin{equation}
\ddt{} \omega_{\ve}  = - \textrm{Ric}(\omega_{\ve}) - \ddbar \log ( (\pi^* \gamma_{\ve}) \Omega_X) + \pi^* \chi.
\end{equation}

We have a lemma.

\pagebreak[3]
\begin{lemma} \label{lemmaf} There exists $C>0$, such that for all $\ve\in (0,\ve_0)$
\begin{enumerate}
\item[(i)] $\displaystyle{\ddbar \log \gamma_\ve \ge - C \omega_Y}$ on $Y$.
\item[(ii)] $\displaystyle{\ddbar \log \gamma_\ve > 0}$ on  $B$.
\item[(iii)] In $\pi^{-1} (B \setminus \{0 \})$ we have
\begin{align} 
\ddt{} \omega_{\ve} & \le - \emph{Ric}(\omega_{\ve}) + (n-k)  \ddbar \log r^2,
\end{align}
where we note that $\ddbar \log r^2 =  \frac{\sqrt{-1}}{2\pi} \frac{1}{r^2} \sum_{i,j} \left( \delta_{ij} - \frac{\overline{z^i} z^j}{r^2} \right) dz^i \wedge d\overline{z^j}$.

\end{enumerate}
\end{lemma}

\begin{proof} Since (i) follows from (ii), we first prove (ii).  Observe that
\begin{equation}
\ddbar \log \gamma_\ve = \frac{(n-k)}{k} \ddbar \log ( \ve + (|z^1|^2 + \cdots + |z^n|^2)^{k}).
\end{equation}
Writing $u=u(\rho) = \log (\ve + e^{k \rho})$, where $\rho = \log r^2$, it suffices to show that $u'>0$, $u''>0$.  Compute
\begin{equation}
u' = \frac{k e^{k \rho}}{\ve + e^{k\rho}} >0, \quad u'' = \frac{\ve k^2 e^{k{\rho}}}{(\ve + e^{k\rho})^2}>0,
\end{equation}
as required.  

In $\pi^{-1} (B \setminus \{0 \})$, using (ii), we have
\begin{eqnarray} \nonumber
- \ddbar \log ( (\pi^* \gamma_{\ve}) \Omega_X) + \pi^* \chi 
& =&  - \ddbar \log ( \pi^* \gamma_{\ve}) + \ddbar \log \frac{\pi^* \Oorb}{\Omega_X} \\
& \le &   \ddbar \log \pi^* r^{2(n-k)},
\end{eqnarray}
and (iii) follows.
\qed
\end{proof}





Next we prove the following estimates:

\begin{lemma} \label{lemmavole}
There exists $C>0$ and $\alpha>0$ such that for all $\ve\in(0,\ve_0)$ and on $[T, T']\times X$,
\begin{enumerate}
\item[(i)] $\displaystyle{\frac{\omega_{\ve}^n}{ \Omega_X}\leq C}.$
\item[(ii)] $\displaystyle{\omega_{\ve} \le \frac{C}{|s|^{2\alpha}_h} \pi^*\omega_Y}$.
\item[(iii)] $\displaystyle{\omega_{\ve}^n \ge \frac{1}{C} |s|^{2(n-k)/k}_h \omega_0^n}$.
\end{enumerate}
\end{lemma}

\begin{proof} Using the volume form bound for the K\"ahler-Ricci flow before the singular time,  note that when $t=T$, 
\begin{equation}
\frac{\omega_\ve^n}{\Omega_X} =  \frac{\omega_{T-\ve}^n}{ \Omega_X}
\end{equation} is uniformly bounded from above, independent of $\ve$.   Using Lemma \ref{lemmaf}, we compute
\begin{equation}
(\ddt{}- \Delta_\ve)\log \frac{\omega_\ve^n}{\Omega_X} =  (\ddt{}- \Delta_\ve) \ddt{\varphi_\ve} - \Delta_\ve \log (\pi^* \gamma_{\ve})= \tr{\omega_{\ve}}{\pi^*(\chi- \ddbar \log  \gamma_{\ve})}\leq C \tr{\omega_\ve}{ \pi^*\omega_Y}.
\end{equation}

We will apply the maximum principle to the quantity 
\begin{equation}
H= \log \frac{\omega_\ve^n}{\Omega_X}  - A\varphi_\ve.
\end{equation}
For $A$ sufficiently large, we have
\begin{equation} \label{Alarge}
A\Delta_{\ve} \varphi_{\ve} = A\tr{\omega_{\ve}}{(\omega_{\ve} - \pi^*( \hat{\omega}_{Y,t} - \frac{\ve}{T} \omega_Y ) - \frac{\ve}{T}\omega_0 )} \le An - C \tr{\omega_\ve}{ \pi^*\omega_Y}.
\end{equation}
Since $\gamma_{\ve}$ and $|\varphi_{\ve}|$ are uniformly bounded from above,
\begin{equation}
(\ddt{}-\Delta_\ve) H \le An + A \log \frac{(\pi^*\gamma_\ve) \Omega_X}{\omega_\ve^n} \le C - A H,
\end{equation}
and thus $H$ is bounded from above by the maximum principle.   This gives (i).

For (ii), we first compute using Lemma \ref{lemmaf},
\begin{eqnarray} \left( \ddt{} - \Delta_{\omega_{\ve}} \right)  \log \tr{\pi^* \omega_Y}{\omega_{\ve}} \le C'  \tr{\omega_{\ve}}{\pi^* \omega_Y} + \frac{C}{|s|^{2\alpha}_h \tr{\pi^*\omega_Y}{\omega_{\ve}}},
\end{eqnarray}
for a uniform constants $C$, $C'$ assuming that $\alpha$ is sufficiently large.  We also have
\begin{eqnarray} \nonumber
\Delta_{\ve} \varphi_{\ve}& = & \tr{\omega_{\ve}}{\left( \omega_{\ve} - \pi^*(\hat{\omega}_{Y,t} - \frac{\ve}{T}\omega_Y) - \frac{\ve}{T} \omega_0\right)} \\
& \le & n - c_0 \, \tr{\omega_{\ve}}{\pi^* \omega_Y}, \label{lapphie}
\end{eqnarray}
for some $c_0>0$.
For a large constant $A$, we define $Q = \log \left( |s|^{2\alpha}_h \tr{\pi^* \omega_Y}{\omega_{\ve}} \right) - A \varphi_{\ve}$.   
As long as $\alpha$ is sufficiently large, $Q$ tends to $-\infty$ along $E$.  We compute on $X \setminus E$,
\begin{eqnarray} \nonumber
\left( \ddt{} - \Delta_{\omega_{\ve}} \right) Q  & \le &  \frac{C}{|s|^{2\alpha}_h \tr{\pi^*\omega_Y}{\omega_{\ve}}}  - \tr{\omega_{\ve}}{\left( A c_0 \pi^*\omega_Y - \alpha R(h) - C' \pi^* \omega_Y \right)} \\
&& \mbox{} - A \log \frac{\omega_{\ve}^n}{(\pi^* \gamma_{\ve}) \Omega_X} + An. 
\end{eqnarray}
We wish to show that $Q$ is bounded from above, and since $\varphi_{\ve}$ is uniformly bounded we may assume without loss of generality that at a maximum point of $Q$,
$\frac{1}{|s|^{2\alpha}_h \tr{\pi^*\omega_Y}{\omega_{\ve}}} \le C''$.
Choosing $A$ sufficiently large we have 
\begin{eqnarray} 
\left( \ddt{} - \Delta_{\omega_{\ve}} \right) Q  & \le & C  - \tr{\omega_{\ve}}{\omega_0} - A \log \frac{\omega_{\ve}^n}{(\pi^* \gamma_{\ve}) \Omega_X} 
 \le  C  - \tr{\omega_{\ve}}{\omega_0} - A \log \frac{\omega_{\ve}^n}{\omega_0^n}. \qquad 
\end{eqnarray}
Hence at a maximum point of $Q$ we have
\begin{equation} 
A \log \frac{\omega_{\ve}^n}{\omega_0^n} + \tr{\omega_{\ve}}{\omega_0} \le C,
\end{equation}
It follows that at the maximum point of $Q$ we have $\tr{\omega_0}{\omega_{\ve}}  \le C$, and hence as long as we chose $\alpha$ sufficiently large,  $\tr{\pi^* \omega_Y}{\omega_{\ve}} \le C |s|^{-2\alpha}_h$.
 It follows that $Q$ is bounded from above at that point and we obtain (ii).

For (iii) put $W = \dot{\varphi}_{\ve} + A \varphi_{\ve}$ for $A$ to be determined.  Compute
\begin{eqnarray}
\left( \ddt{} - \Delta_{\omega_{\ve}} \right) W & = & \tr{\omega_{\ve}} \pi^*\chi + A \dot{\varphi}_{\ve} - A \tr{\omega_{\ve}}(\omega_{\ve} - (\pi^*(\hat{\omega}_{Y,t} - \frac{\ve}{T} \omega_Y) + \frac{\ve}{T} \omega_0)).
\end{eqnarray}
Then choosing $A$ large enough so that
\begin{equation}
A(\pi^*(\hat{\omega}_{Y,t} - \frac{\ve}{T} \omega_Y) + \frac{\ve}{T} \omega_0) \ge - \pi^* \chi,
\end{equation}
 we obtain
\begin{equation}
\left( \ddt{} - \Delta_{\omega_{\ve}} \right) W  \ge  A \dot{\varphi}_{\ve} - An \ge AW - C,
\end{equation}
for a constant $C$.  It follows from the maximum principle that $W$ and hence $\dot{\varphi}_{\ve}$ is bounded from below, giving (iii). \qed
\end{proof}

Note that Lemma \ref{lemmavole} shows in particular  that $\omega_{\ve}$ is uniformly bounded from above and below away from zero on compact subsets of $X \setminus E$.
We now prove an upper bound for $\omega_{\ve}$ near $E$, which is analogous to parts (i) and (ii) of Lemma \ref{lemmaest2}.

\begin{lemma} \label{lemma25}
There exists $\delta>0$ and a uniform constant $C$ such that on $X \setminus E$,
\begin{enumerate}
\item[(i)] $\displaystyle{\omega_{\ve} \le \frac{C}{|s|^{2/k}_h}  \pi^*\omega_{\emph{Eucl}}}$.
\item[(ii)] $\displaystyle{\omega_{\ve} \le \frac{C}{|s|^{2(1-\delta)/k}_h} (\omega_0} + \pi^*\omega_{\emph{Eucl}})$.
\end{enumerate}
\end{lemma}
\begin{proof}
It suffices to prove the estimates (i) and (ii) on a small neighborhood $V$ of $E$ in $X$, contained in $\pi^{-1}(B)$.
Applying Lemma \ref{lemmaf} and using the fact that $\oeuc$ is flat on $B$,  we have on $V \setminus E$,
\begin{eqnarray} 
\left( \ddt{} - \Delta_{\omega_{\ve}} \right) \log \tr{\pi^* \oeuc}{\omega_{\ve}} & \le & \frac{C}{|s|^{2/k}_h \tr{\pi^* \oeuc}{\omega_{\ve}}}.
\end{eqnarray}
and using (\ref{hatg}),
\begin{eqnarray} \nonumber
\left( \ddt{} - \Delta_{\omega_{\ve}} \right) \log \tr{\omega_0}{\omega_{\ve}}  & \le & C \tr{\omega_{\ve}}{\omega_0} + \frac{(n-k)}{\tr{\omega_0}{\omega_{\ve}}} \tr{\omega_0}{\left(\frac{\sqrt{-1}}{2\pi} \frac{1}{r^2} \sum_{i,j} \left( \delta_{ij} - \frac{\overline{z^i} z^j}{r^2} \right) dz^i \wedge d\overline{z^j}\right)} \\
& \le & C \tr{\omega_{\ve}}{\omega_0} + \frac{C}{\tr{\omega_0}{\omega_{\ve}}},
\end{eqnarray}
since $\omega_0 \ge \frac{c}{r^2}(\delta_{ij} - \frac{\ov{z^i}z^j}{r^2} )$ for some $c>0$.  Fix a small $\eta>0$.  We will apply the maximum principle to the quantity
\begin{equation}
Q_{\eta} = \log \tr{\omega_0}{\omega_{\ve}} + A \log( |s|_h^{2(1+ \eta)/k} \tr{\pi^* \oeuc}{\omega_{\ve}}) - A \varphi_{\ve}.
\end{equation}
Then, using (\ref{lapphie})
\begin{eqnarray} \nonumber
\left( \ddt{} - \Delta_{\omega_{\ve}} \right) Q_{\eta}  & \le & C  \tr{\omega_{\ve}}{\omega_0}+ \frac{C}{\tr{\omega_0}{\omega_{\ve}}} + \frac{CA}{|s|^{2/k}_h (\tr{\pi^* \oeuc}{\omega_{\ve}})} +A n \\ 
&& \mbox{} - A \tr{\omega_{\ve}}{\left(c_0 \pi^*\omega_Y - \frac{(1+\eta)}{k} R(h)  \right)}
  - A \log \frac{\omega_{\ve}^n}{(\pi^*\gamma_{\ve}) \Omega_X}  . \label{eQ}
\end{eqnarray}
We wish to show that at a maximum point of $Q_{\eta}$ we have $\tr{\omega_0}{\omega_{\ve}} \le C$.  It would then follow that  $Q_{\eta} \le C$, since
from (\ref{orbineq}) we have
\begin{equation} \label{orbineq2}
|s|^{2/k}_h \tr{\pi^* \oeuc}{\omega_{\ve}} \le \tr{\omega_0}{\omega_{\ve}}.
\end{equation}
Hence we may assume without loss of generality that $\tr{\omega_0}{\omega_{\ve}} \ge 1$
and, since $\varphi_{\ve}$ is uniformly bounded,
\begin{equation}
\frac{1}{|s|^{2/k}_h (\tr{\pi^* \oeuc}{\omega_{\ve}})} \le (\tr{\omega_0}{\omega_{\ve}})^{1/A}.
\end{equation}
Applying the volume form bound of Lemma \ref{lemmavole}, we may assume then that, taking $A>n-1$,
\begin{equation}
\frac{CA}{|s|^{2/k}_h (\tr{\pi^* \oeuc}{\omega_{\ve}})} \le CA(\tr{\omega_{\ve}}{\omega_0})^{(n-1)/A} \le C_A + \tr{\omega_{\ve}}{\omega_0},
\end{equation}
where $C_A$ means a constant that depends on $A$.
Hence, by  (\ref{eQ}), we have at a maximum of $Q_{\eta}$,
\begin{equation} 
0   \le  (C+1)  \tr{\omega_{\ve}}{\omega_0} - A\tr{\omega_{\ve}}{\left( c_0 \pi^*\omega_Y - \frac{(1+\eta)}{k} R(h) \right)} \label{eQ2}
  - A \log \frac{\omega_{\ve}^n}{(\pi^*\gamma_{\ve}) \Omega_X}  + C_A. \
\end{equation}
Choosing $A$ sufficiently large
we have at  a maximum point of $Q_{\eta}$,
\begin{equation} 
A \log \frac{\omega_{\ve}^n}{(\pi^*\gamma_{\ve}) \Omega_X} + \tr{\omega_{\ve}}{\omega_0} \le C,
\end{equation}
and hence
\begin{equation} 
A \log \frac{\omega_{\ve}^n}{\omega_0^n} + \tr{\omega_{\ve}}{\omega_0} \le C,
\end{equation}
giving $\tr{\omega_0}{\omega_{\ve}}  \le C$ as required.  Hence by the maximum principle we have a uniform upper bound on $Q_{\eta}$, independent of $\eta$.  Then (i) follows by applying again (\ref{orbineq2}).  Finally, (ii) follows by a similar argument to that in Lemma \ref{lemmaest2}.  \qed
\end{proof}

We can now prove:

\begin{lemma}
There exist positive constants $C$, $\alpha$, independent of $\ve$, such that on $X \setminus E$,
\begin{equation} \label{oe1}
\frac{|s|^{2\alpha}}{C} \omega_0 \le \omega_{\ve} \le \frac{C}{|s|^{2\alpha}_h} \omega_0.
\end{equation}
Fix a large positive integer $N$.  Then for each integer $m$ with $0\le m \le N$ there exist $C_m, \alpha_m>0$ such that
\begin{equation} \label{oe2}
| (\nabla^0_{\mathbb{R}})^m g_{\ve} |_{\omega_0} \le \frac{C_m}{|s|^{2\alpha_m}_h},\\
\end{equation}
where $\nabla^0_{\mathbb{R}}$ denotes the real covariant derivative with respect to $g_0$.
\end{lemma}
\begin{proof}
The inequalities (\ref{oe1}) follow immediately from Lemmas \ref{lemmavole} and \ref{lemma25}.  For the higher order estimates (\ref{oe2}) we apply a similar argument as in \cite{SW2}.  \qed
\end{proof}

Thus we have estimates of $\varphi_{\ve}$ of all orders away from $E$.  We can take a convergent subsequence of  $\varphi$.  Setting $\tilde{\varphi}$ to be the `push down' of $\varphi$ on $Y\setminus \{ y_0\}$, we have proved the following.

\begin{proposition} $\tilde{\varphi}$ constructed above solves the parabolic Monge-Amp\`ere flow
\begin{equation}
\ddt{ \tilde{\varphi}}= \log \frac{ (\hat{\omega}_{Y,t}+\ddbar\tilde{\varphi})^n}{\Omega_{\emph{orb}}}, ~~~\tilde{\varphi}|_{t=T}= \psi_T
\end{equation}
on $[T, T']\times Y\setminus \{y_0\}$ with
with $\tilde{\varphi} \in L^\infty ([T,T']\times Y) \cap C^\infty ([T, T']\times Y\setminus \{y_0\}).$
\end{proposition}

We can now finish the proof of Theorem \ref{mainthm}.

\begin{proof}[Proof of part (v) of Theorem \ref{mainthm}]  Apply Lemma \ref{lemmau} to see that $\tilde{\varphi}$ coincides with the solution $\varphi$ constructed in Section \ref{sectiv}.  Using Lemma \ref{lemma25}, there exist uniform constants $C$ and $\delta>0$ so that the solution $\omega(t)$ of the K\"ahler-Ricci flow on $Y$ constructed in Section \ref{sectiv} satisfies the estimates
\begin{enumerate}
\item[(i)] $\displaystyle{\omega(t) \le \frac{C}{r^2}  \oeuc}$
\item[(ii)] $\displaystyle{\omega(t) \le \frac{C}{r^{2(1-\delta)}} (\omega_0} + \oeuc)$
\item[(iii)] $\displaystyle{
|W|^2_{g} \le \frac{C}{r^{2(1-\delta)}}}$ 
where $W = \sum_{i=1}^n \left( \frac{x_i}{r} \frac{\partial}{\partial x_i} + \frac{y_i}{r} \frac{\partial}{\partial y_i} \right)$ the unit length radial vector field with respect to $g_{\textrm{Eucl}}$, with $z_i = x_i + \sqrt{-1} y_i$.
\end{enumerate}
It follows as in \cite{SW2} that $(Y,g(t))$ converges to $(Y, d_T)$ in the Gromov Hausdorff sense as $t \rightarrow T^+$. \qed
\end{proof}

\section{Application to the manifolds $M_{n,k}$} \label{sectbundles}

In this section we provide more details of the statement of  Theorem \ref{thmhirz} and show how it follows from Theorem \ref{mainthm}.

The complex manifold $M_{n,k}$ can be described as 
\begin{align} \nonumber
M_{n,k}  & =  \{ ([Z_1, \ldots, Z_n], (\sigma, \mu)) \in \mathbb{P}^{n-1} \times ( (\mathbb{C}^n \times \mathbb{C}) \setminus \{ (0,0) \}) \ | \ \sigma \textrm{ is in} \\ & \quad  \ \textrm{ the line } \lambda \mapsto (\lambda (Z_1)^k, \ldots, \lambda (Z_n)^k) \}/ \sim
\end{align}
where
\begin{equation}
([Z_1, \ldots, Z_n], (\sigma, \mu)) \sim ([Z_1, \ldots, Z_n], (\sigma', \mu'))
\end{equation}
if there exists $a \in \mathbb{C}^*$ such that $(\sigma, \mu) = (a \sigma', a \mu')$.  Then $D_0$ and $D_{\infty}$ are the divisors $\{ \sigma =0 \}$ and 
$\{ \mu =0 \}$ respectively.

We define the map $\pi: M_{n,k} \rightarrow Y_{n,k}$ by
\begin{align} \label{defpi}
\pi\left( ([Z_1, \ldots, Z_n], (\sigma, \mu) )\right) = [\mu, b^{1/k} Z_1, \ldots, b^{1/k} Z_n],
\end{align}
where $b \in \mathbb{C}$ is defined by
\begin{align}
\sigma = b ((Z_1)^k, \ldots, (Z_n)^k).
\end{align}
Note that $b^{1/k}$ is only determined up to multiplication by a $k$th root of unity, but by the definition of $Y_{n,k}$,  the element
$[\mu, b^{1/k} Z_1, \ldots, b^{1/k} Z_n]$ in $Y_{n,k}$ is uniquely defined.  Moreover $\pi$ is globally well-defined, surjective, and injective on the complement of $D_0$.  

If we identify the line bundle $\mathcal{O}(-k)$ with the open subset $\{ \mu \neq 0 \}$ of $M_{n,k}$ and $\mathbb{C}^n/\mathbb{Z}_k$ with the open subset $\{ Z_0 \neq 0\}$ of $Y_{n,k}$ via $(z_1, \ldots, z_n) \mapsto [1, z_1, \ldots, z_n]$, then $\pi$ coincides with the map described in Section \ref{sectlocalmodel}.  Thus $\pi : M_{n,k} \rightarrow Y_{n,k}$ is precisely the blow down of the $(-k)$ exceptional divisor $D_0$ to the orbifold point $[1,0, \ldots, 0]$ in $Y_{n,k}$.

It remains to check condition (\ref{condition}).  The K\"ahler class $\alpha_t$ of $\omega(t)$ is given by
\begin{align}
\alpha_t = \frac{b_t}{k} [D_{\infty}] - \frac{a_t}{k} [ D_0],
\end{align}
where 
\begin{align}
b_t = b_0 - (k+n) t \quad \textrm{and} \quad a_t = a_0 + (k-n)t.
\end{align}
Under the assumptions (a) and (b) of Theorem \ref{thmhirz} we see that the inequalities $0<a_t<b_t$ hold until time $T = a_0/(n-k)$ and the limiting class is 
\begin{align}
[\omega_0] + T c_1(K_{M_{n,k}}) = \alpha_T = \frac{b_T}{k} [D_{\infty}],
\end{align}
where $b_T = b_0 - (k+n) a_0/ (n-k)>0$.

Consider the potential function $u$ on $\mathbb{C}^n$ given by 
\begin{align}
u = \log (1 +  ( | z^1|^2 + \cdots |z^n|^2 )^k).
\end{align}
The $(1,1)$ form $\omega_u = \ddbar u$ is $\mathbb{Z}_k$-invariant and defines a  (1,1) orbifold form on $\mathbb{C}^n/\mathbb{Z}_k$.  Moreover, $\omega_u$ extends to be a closed nonnegative orbifold $(1,1)$ form $\omega_Y$ on the weighted projective space $Y_{n,k}$, which is positive definite outside of the orbifold point (see \cite{C, FIK, SW1}).  The pull-back $\pi^* \omega_Y$ is smooth and lies in the cohomology class $[D_{\infty}]$.  We can choose a Hermitian metric on the line bundle $[D_0]$ over $M_{n,k}$ so that  $\pi^* \omega_Y - c R(h)$ is K\"ahler for $c>0$ sufficiently small.
 Hence
\begin{align}
[\omega_0] + T c_1(K_{M_{n,k}}) = \frac{b_T}{k} [\pi^* \omega_Y],
\end{align}
and $\omega_Y$ satisfies the conditions of Lemma \ref{lemmaoy}.  By the remark at the end of Section \ref{sectlocalmodel} we can apply Theorem \ref{mainthm} to obtain Theorem \ref{thmhirz}.

\section{Minimal surfaces of general type}
We now give a proof of Theorem \ref{thmmin}.  We assume for simplicity that $X$ is a minimal surface of general type with a single, irreducible, $(-2)$-curve $C$.

 As discussed in the introduction, we have a holomorphic map $\Phi: X \rightarrow \mathbb{P}^N$ whose image is $Y= \Xc$ and which blows down $C$.  Write $y_0 = \Phi(C) \in Y$.   

Let $\omega(t)$ be a solution of the normalized K\"ahler-Ricci flow (\ref{krf1}) for $t \in [0,\infty)$.   Denote by $\ofs$   the Fubini-Study metric on $\mathbb{P}^N$, and write $\chi = \Phi^* \ofs$.  Observe that $\chi$ is a smooth nonnegative (1,1) form on $X$.  Let $\Omega$ be a smooth volume form on $X$ satisfying
\begin{equation}
\chi = \frac{\sqrt{-1}}{2\pi} \partial \ov{\partial} \log \Omega \ge 0,
\end{equation}
and define a family of reference metrics $\hat{\omega}_t$ by
\begin{equation}
\hat{\omega}_t = e^{-t} \omega_0 + (1-e^{-t}) \chi \in [\omega(t)].
 \end{equation}
 
We write the normalized K\"ahler-Ricci flow as a parabolic complex Monge-Amp\`ere flow of potentials.  If $\varphi= \varphi(t)$ solves
\begin{equation}
 \ddt{\varphi} = \log \frac{ (\hat{\omega}_t + \ddbar \varphi)^n }{\Omega} - \varphi, ~~~~~~~~\varphi|_{t=0} = 0,
 \end{equation}
then $\omega(t) = \hat{\omega}_t + \frac{\sqrt{-1}}{2\pi} \partial \ov{\partial} \varphi$ solves the normalized K\"ahler-Ricci flow (\ref{krf1}) with initial metric $\omega_0$.

It was shown by Tsuji \cite{Ts} and Tian-Zhang \cite{TZha} that:

\begin{theorem} \label{ttz} Let $\omega_{\emph{KE}}$ be the smooth orbifold Kahler-Einstein metric on $X_{\emph{can}}$. Then $\omega(t)$ converges to  $\Phi^* \omega_{\emph{KE}}$ in $C^{\infty}$ on compact subsets of  $X\setminus C$ as $t \rightarrow \infty$.
\end{theorem}

Moreover, the following result was proved by Zhang \cite{Zha2}.

\begin{theorem} There exist uniform constants $c,C>0$ such that for all $t\geq0$,
\begin{enumerate}
\item[(i)] $\displaystyle{||\varphi||_{L^\infty}\leq C}$.
\item[(ii)] $\displaystyle{ c\, \omega_0^2 \leq \omega^2  \leq C \omega_0^2.}$
\item[(iii)] $| R| \le C$, where $R=R(\omega)$ is the scalar curvature.
\end{enumerate}
\end{theorem}

From Theorem \ref{ttz} we see that the metric $\omega(t)$ is uniformly bounded on compact subsets of $X \setminus C$.
It remains then to understand the behavior of the metric $\omega(t)$ in a neighborhood of $C$.

The local model for $X$ in a neighborhood $U$ of $C$ is given by a neighborhood of the $(-2)$-curve $D_0$ in the second Hirzebruch surface $M_{2,2}$.  Indeed $y_0$ is a double point of $X_{\textrm{can}}$, and the minimal resolution of singularities for surfaces is unique (see Theorems 6.1, 6.2 in \cite{BHPV}).   The blow-down map $\Phi$ restricted to $U$ can be identified with the map $\pi: M_{2,2} \rightarrow Y_{2,2}$ restricted to a neighborhood of $D_0$.  From now on, we make these identifications.

Hence there exists a bounded neighborhood $B$ of the origin in $\mathbb{C}^2$ and a $2$-to-$1$ map from $B \setminus \{ 0 \}$ to $U \setminus C$.  We will write $\omega=\omega(t)$ for the evolving K\"ahler metric pulled back via this map to $B \setminus \{0 \}$.  We wish to obtain bounds for $\omega(t)$ on $B \setminus \{ 0\}$.  As in Section \ref{secthirz} we will write $\oeuc$ for the Euclidean metric on $\mathbb{C}^2$.  Write $z_1,z_2$ for the coordinates on $B \subset \mathbb{C}^2$, and $r^2 = |z_1|^2 + |z_2|^2$.  We have the following result.

\pagebreak[3]
\begin{lemma} \label{lemmaest3}
There exist uniform constants $C$ and $\delta>0$ such that on $B\setminus\{0\}$,
\begin{enumerate}
\item[(i)] $\displaystyle{\omega \le \frac{C}{r^2} \, \omega_{\emph{Eucl}}.}$
\item[(ii)]$\displaystyle{\omega \le \frac{C}{r^{2(1-\delta)}} (\omega_0 + \omega_{\emph{Eucl}}),}$
\item[(iii)] $\displaystyle{
|W|^2_{g(t)} \le \frac{C}{r^{2/3}}}$ 
where $W = \sum_{i=1}^2 \left( \frac{x_i}{r} \frac{\partial}{\partial x_i} + \frac{y_i}{r} \frac{\partial}{\partial y_i} \right)$ the unit length radial vector field with respect to $g_{\emph{Eucl}}$, with $z_i = x_i + \sqrt{-1} y_i$.
\end{enumerate}
\end{lemma}
\begin{proof}
The proof of this result is essentially the same as the proof of Lemma \ref{lemmaest2}.   Observe that $\omega_0$ on $B$ is equivalent to the metric $\hat{g}$ given by (\ref{hatg}) with $k=2$.  Hence by the same argument as in the proof of Lemma \ref{lemmahato} we obtain
$\omega_0 \le \frac{C}{r^2} \oeuc,$

  From the evolution equation (\ref{krf1}), the analogue of (\ref{tr1}) is
\begin{eqnarray} 
\left( \ddt{} - \Delta \right) \log \tr{\tilde{\omega}}{\omega} 
& \le &   \frac{1}{\tr{\tilde{\omega}}{\omega}  } \left( - \tr{\tilde{\omega}}{\omega} - g^{i \ov{j}} \tilde{R}_{i \ov{j}}^{ \ \ k \ov{\ell}} g_{k \ov{\ell}}\right) \le -1+ \tilde{C} \, \tr{\omega}{\tilde{\omega}}. \label{tr2}
\end{eqnarray}
For (i) and (ii), we evolve the same quantity as in the proof of Lemma \ref{lemmaest2}.  The only difference is that we replace the result of Lemma \ref{lemmaoy} by the following argument.   Since $[\chi]=[K_X]$ is big and nef, then by Kodaira's lemma (cf. \cite{Ts}, for example), there exists a Hermitian metric $h$ on the line bundle $[C]$ and a constant $\ve_0$ such that for all $0 < \ve \le \ve_0$,
 \begin{equation} \label{kod}
 \chi - \varepsilon R(h) >0.
 \end{equation}
It follows that there exists a constant $c>0$ and $\ve'>0$ such that
  \begin{equation} \label{kod2}
\hat{\omega}_t - \varepsilon' R(h) > c\, \omega_0,
 \end{equation}
 for all $t \ge 0$.  The rest of the argument for (i) and (ii) follows in the same way as in Lemma \ref{lemmaest2}.
 
 For (iii) we evolve the quantity $Q_{\varepsilon} = \log ( |V|_{\omega}^{(1+\varepsilon)} \tr{\oeuc}{\omega})$ and make use again of (\ref{tr2}) with $\tilde{\omega}= \oeuc$.  This finishes the proof of the lemma.
 \qed
\end{proof}

\begin{proof}[Proof of Theorem \ref{thmmin}]  This result follows from Theorem \ref{ttz} together with the radial bound given by Lemma \ref{lemmaest3}, part (iii).  Indeed, we can show that for every $\ve>0$ there exists $T=T(\ve)$ and $r = r({\ve})>0$ such that the diameter of a ball $B_r$ with respect to $g(t)$ for $t\ge T$ is less than $\ve$.  To see this note that we can find $T$  and $r$ such that
\begin{enumerate}
\item[(a)]  the diameter of $B_r \setminus B_{r/2}$ with respect to $g(t)$ for $t \ge T$ is less than $\ve/3$, since $g(t)$ converges to the K\"ahler-Einstein metric uniformly on compact subsets of $Y \setminus \{ y_0\}$; and
\item[(b)]  the length of any radial line in $B_r$ is less than $\ve/3$.
\end{enumerate}
For any $p, q \in B_r$ there exist $p', q' \in B_r \setminus B_{r/2}$ connected to $p$ and $q$ by radial lines.  Hence
\begin{equation}
d_{g(t)} (p,q) \le d_{g(t)}(p,p') + d_{g(t)}(p',q') + d_{g(t)}(q',q) < \ve,
\end{equation}
as required.

It is now a straightforward matter to prove the Gromov-Hausdorff convergence of $(X, g(t))$ to $(\Xc, g_{\textrm{KE}})$. \qed
\end{proof}

We end by remarking that there are plenty of examples of minimal surfaces of the type considered in Theorem \ref{thmmin}.  Indeed, let $C$ be a hyperelliptic curve defined by $y=f(x)$, where $f(x)$ has degree $n>4$ with $n$ distinct roots $p_1, ..., p_n$. Then $y\mapsto -y$ gives a $\mathbb{Z}_2$-action with fixed points as $(y=0, x=p_i)$ and those at $\infty$. Define $X = (C \times C) / \mathbb{Z}_2$, where $\mathbb{Z}_2$ acts diagonally (that is, $(-1)\cdot (z_1,z_2 )= ( (-1)\cdot z_1, (-1)\cdot z_2)$). Then $q_i=\{p_i\} \times \{p_i\}$ are the orbifold points of $X$ with structure group $\mathbb{Z}_2$.
 $X$ is an orbifold surface of general type, as can be seen by pulling back the K\"ahler-Einstein metric on $C\times C$.
Now consider $\pi: \hat X \rightarrow X$ the minimal resolution of $X$. Since each $q_i$ is an $A_1$ singularity, $\pi^{-1}(q_i)$ is an $A_1$ curve. $\hat X$ is a minimal surface of general type with only distinct irreducible $(-2)$ curves. 

\bigskip
\noindent
{\bf Acknowledgements} \  The authors thank Tom Ilmanen, Yuguang Zhang and Zhou Zhang for some helpful discussions.

\bigskip
\bigskip

$^{*}$ Department of Mathematics \\
Rutgers University, Piscataway, NJ 08854\\

$^{\dagger}$ Department of Mathematics \\
University of California San Diego, La Jolla, CA 92093


\begin{thebibliography}{99}
\bibitem[A]{A} Aubin, T.  {\em \'Equations du type Monge-Amp\`ere sur les vari\'et\'es k\"ahl\'eriennes compactes},  Bull. Sci. Math. (2) {\bf 102} (1978), no. 1, 63--95
\bibitem[BHPV] {BHPV} Barth, W. P., Hulek, K., Peters, C. A. M. and Van de Ven, A. {\em Compact complex surfaces.} Second edition. Ergebnisse der Mathematik und ihrer Grenzgebiete.  Springer-Verlag, Berlin
\bibitem[C]{C} Calabi, E. {\em  Extremal K\"ahler metrics}, in Seminar on Differential Geometry,  pp. 259--290, Ann. of Math. Stud., {\bf 102}, Princeton Univ. Press, Princeton, N.J., 1982
\bibitem[Cao]{Cao} Cao, H.-D. {\em Deformation of K\"ahler metrics to K\"ahler-Einstein metrics on compact
K\"ahler manifolds},  Invent. Math. {\bf 81} (1985),  no. 2, 359--372
\bibitem[CC]{CC} Cheeger, J. and Colding, T. H. {\em On the structure of spaces with Ricci curvature bounded below. I}, J. Differential Geom. {\bf 45} (1997), 406--480
\bibitem[CW]{CW} Chen, X. and Wang, B. {\em K\"ahler-Ricci flow on Fano manifolds (I)}, preprint, arXiv: 0909.2391
\bibitem[EGZ]{EGZ} Eyssidieux, P., Guedj, V. and Zeriahi, A. {\em Singular K\"ahler-Einstein metrics}. J. Amer. Math. Soc. {\bf 22} (2009), no. 3, 607--639
\bibitem[FIK]{FIK} Feldman, M., Ilmanen, T. and Knopf, D. {\em Rotationally symmetric shrinking and expanding gradient K\"ahler-Ricci solitons},  J. Differential Geometry {\bf 65}  (2003),  no. 2, 169--209
\bibitem[H]{H} Hamilton, R. S. {\em Four-manifolds with positive isotropic curvature},  Comm. Anal. Geom.  {\bf 5}  (1997),  no. 1, 1--92
\bibitem[Kob]{Kob} Kobayashi, R., {\em Einstein-K\"ahler V-metrics on open Satake V-surfaces
with isolated quotient singularities}, Math. Ann. {\bf 272} (1985), no. 3, 385--398
\bibitem[Kol1]{Kol1} Ko{\l}odziej, S. {\em The complex Monge-Amp\`ere equation}, Acta Math. {\bf 180} (1998), no. 1, 69--117
\bibitem[Kol2]{Kol2} Ko{\l}odziej, S. {\em The Monge-Amp\`ere equation on compact K\"ahler manifolds}, Indiana Univ. Math. J. 52 (2003), no. 3, 667--686
\bibitem[LT]{LT} La Nave, G. and Tian, G. {\em Soliton-type metrics and K\"ahler-Ricci flow on symplectic quotients}, preprint, arXiv: 0903.2413.
\bibitem[MS]{MS} Munteanu, O. and Sz\'ekelyhidi, G. {\em On convergence of the K\"ahler-Ricci flow}, preprint, arXiv: 0904.3505
\bibitem[P1]{P1} Perelman, G. {\em The entropy formula for the Ricci flow and its geometric applications},
preprint, arXiv: math.DG/0211159
\bibitem[P2]{P2} Perelman, G. unpublished work on the K\"ahler-Ricci flow
\bibitem[PSSW]{PSSW} Phong, D.H., Song, J., Sturm, J. and Weinkove, B. {\em The K\"ahler-Ricci flow and the $\bar\partial$ operator on vector fields}, J. Differential Geometry {\bf 81} (2009), no. 3, 631--647
\bibitem[PS]{PS} Phong, D.H. and Sturm, J.  {\em On stability and the convergence of the K\"ahler-Ricci flow},  J. Differential Geometry  {\bf 72} (2006),  no. 1, 149--168
\bibitem[SeT]{SeT} Sesum, N. and Tian, G. {\em Bounding scalar curvature and diameter along the K\"ahler Ricci flow (after Perelman)},  J. Inst. Math. Jussieu  {\bf 7}  (2008),  no. 3, 575--587
\bibitem[So]{So} Song, J. {\em Finite time extinction of the K\"ahler-Ricci flow}, preprint, arXiv: 0905.0939
\bibitem[SSW]{SSW}, Song, J., Sz\'ekelyhidi, G. and Weinkove, B. {\em The K\"ahler-Ricci flow on projective bundles}, to appear in Int. Math. Res. Not.
\bibitem[SoT1]{SoT1} Song, J, and Tian, G. {\em The K\"ahler-Ricci flow on surfaces of positive Kodaira dimension},
Invent. Math. {\bf 170} (2007), no. 3, 609--653
\bibitem[SoT2]{SoT2} Song, J, and Tian, G. {\em Canonical measures and K\"ahler-Ricci flow}, preprint, arXiv: 0802.2570
\bibitem[SoT3]{SoT3} Song, J, and Tian, G. {\em  The K\"ahler-Ricci flow through singularities}, preprint, arXiv: 0909.4898
\bibitem[SW1]{SW1} Song, J. and Weinkove, B. {\em The K\"ahler-Ricci flow on Hirzebruch surfaces}, to appear in J. Reine Angew. Math., arXiv: 0903.1900
\bibitem[SW2]{SW2} Song, J. and Weinkove, B.  {\em Contracting exceptional divisors by the K\"ahler-Ricci flow}, to appear in Duke Math. J., arXiv:1003.0718
\bibitem[Sz]{Sz} Sz\'ekelyhidi, G. {\em The K\"ahler-Ricci flow and K-polystability},  Amer. J. Math. {\bf 132} (2010), no. 4, 1077--109
\bibitem[T]{T} Tian, G. {\em New results and problems on K\"ahler-Ricci flow}, G\'eom\'etrie diff\'erentielle, physique math\'ematique, math\'ematiques et soci\'et\'e. II.  Ast\'erisque  {\bf 322}  (2008), 71--92
\bibitem[TZha]{TZha} Tian, G. and Zhang, Z. {\em On the K\"ahler-Ricci flow on projective manifolds of general type},  Chinese Ann. Math. Ser. B  {\bf 27}  (2006),  no. 2, 179--192
\bibitem[TZhu]{TZhu} Tian, G. and Zhu, X. {\em Convergence of K\"ahler-Ricci flow}, J. Amer. Math. Soc. {\bf 20} (2007), no. 3, 675--699
\bibitem[To]{To} Tosatti, V. {\em K\"ahler-Ricci flow on stable Fano manifolds},  J. Reine Angew. Math. {\bf 640} (2010), 67--84
\bibitem[Ts]{Ts} Tsuji, H. {\em Existence and degeneration of K\"ahler-Einstein metrics on minimal algebraic varieties of general type}, Math.
Ann. {\bf 281} (1988), 123--133
\bibitem[Y]{Y1} Yau, S.-T. {\em On the Ricci curvature of a compact K\"ahler
manifold and the complex Monge-Amp\`ere equation, I}, Comm. Pure
Appl. Math. {\bf 31} (1978), 339--411
\bibitem[Zha1]{Zha1} Zhang, Z. {\em On degenerate Monge-Amp\`ere equations over closed K\"ahler manifolds},   Int. Math. Res. Not. {\bf 2006}, Art. ID 63640, 18 pp
\bibitem[Zha2]{Zha2} Zhang, Z. {\em Scalar curvature bound for K\"ahler-Ricci flows over minimal manifolds of general type},  Int. Math. Res. Not. IMRN  {\bf 2009},  no. 20, 3901--3912
\bibitem[Zha3]{Zha3} Zhang, Z. {\em Ricci lower bound for K\"ahler-Ricci flow}, preprint, arXiv: 1110.5954
\bibitem[Zhu]{Zhu} Zhu, X. {\em K\"ahler-Ricci flow on a toric manifold with positive first Chern class},
preprint,  arXiv:math.DG/0703486
\end{thebibliography}
\end{document}